\documentclass[11pt,american,english]{article}
\usepackage[T1]{fontenc}
\usepackage[latin9]{inputenc}
\usepackage[a4paper]{geometry}
\usepackage{babel}
\usepackage{amsmath}
\usepackage{amsthm}
\usepackage{amssymb}
\usepackage{setspace}
\usepackage{esint}
\PassOptionsToPackage{normalem}{ulem}
\usepackage{ulem}
\usepackage[unicode=true,pdfusetitle,
 bookmarks=true,bookmarksnumbered=true,bookmarksopen=true,bookmarksopenlevel=3,
 breaklinks=false,pdfborder={0 0 1},backref=false,colorlinks=false]
 {hyperref}

\makeatletter
\theoremstyle{plain}
\newtheorem{thm}{\protect\theoremname}
  \theoremstyle{plain}
  \newtheorem{prop}[thm]{\protect\propositionname}
  \theoremstyle{remark}
  \newtheorem{rem}[thm]{\protect\remarkname}
  \theoremstyle{plain}
  \newtheorem{lem}[thm]{\protect\lemmaname}

\usepackage{multirow}

\usepackage{arydshln}

\usepackage{lmodern}
\newcommand{\1}{\mbox{1\hspace{-1mm}I}}
\numberwithin{equation}{section}

\usepackage{calc}
\makeatother

  \addto\captionsamerican{\renewcommand{\lemmaname}{Lemma}}
  \addto\captionsamerican{\renewcommand{\propositionname}{Proposition}}
  \addto\captionsamerican{\renewcommand{\remarkname}{Remark}}
  \addto\captionsamerican{\renewcommand{\theoremname}{Theorem}}
  \addto\captionsenglish{\renewcommand{\lemmaname}{Lemma}}
  \addto\captionsenglish{\renewcommand{\propositionname}{Proposition}}
  \addto\captionsenglish{\renewcommand{\remarkname}{Remark}}
  \addto\captionsenglish{\renewcommand{\theoremname}{Theorem}}
  \providecommand{\lemmaname}{Lemma}
  \providecommand{\propositionname}{Proposition}
  \providecommand{\remarkname}{Remark}
\providecommand{\theoremname}{Theorem}

\begin{document}
\selectlanguage{american}%
\global\long\def\1{\mbox{1\hspace{-1mm}I}}

\selectlanguage{english}%

\title{CONTROL PROBLEM ON SPACE OF RANDOM VARIABLES AND MASTER EQUATION }

\author{Alain Bensoussan\\
International Center for Decision and Risk Analysis\\
Jindal School of Management, University of Texas at Dallas\thanks{Also with the College of Science and Engineering, Systems Engineering
and Engineering Management, City University Hong Kong. Research supported
by the National Science Foundation under grant DMS-1303775 and the
Research Grants Council of the Hong Kong Special Administrative Region
(CityU 500113). } \\
Sheung Chi Phillip Yam\\
Department of Statistics, The Chinese University of Hong Kong\thanks{Research supported by The Hong Kong RGC GRF 14301015 with the project
title: Advance in Mean Field Theory.}}
\maketitle
\begin{abstract}
We study in this paper a control problem in a space of random variables.
We show that its Hamilton Jacobi Bellman equation is related to the
Master equation in Mean field theory. P.L. Lions in (\cite{PLL}),(\cite{PL2}
introduced the Hilbert space of square integrable random variables
as a natural space for writing the Master equation which appears in
the mean field theory. W. Gangbo and A. \'{S}wi\k{e}ch \cite{GAS}
considered this type of equation in the space of probability measures
equipped with the Wasserstein metric and use the concept of Wasserstein
gradient. We compare the two approaches and provide some extension
of the results of Gangbo and \'{S}wi\k{e}ch.
\end{abstract}

\section{INTRODUCTION }

We study first an abstract control problem where the state is in a
Hilbert space. We then show how this model applies when the Hilbert
space is the space of square integrable random variables, and for
certain forms of the cost functions. We see that it applies directly
to the solution of the Master equation in Mean Field games theory.
We compare our results with those of W. Gangbo and A. \'{S}wi\k{e}ch
\cite{GAS} and show that the approach of the Hilbert space of square
integrable random variables simplifies greatly the development.

\section{AN ABSTRACT CONTROL PROBLEM}

\subsection{\label{sub:SETTING-OF-THE}SETTING OF THE PROBLEM}

We begin by defining an abstract control problem, without describing
the application. We consider a Hilbert space $\mathcal{H}$, whose
elements are denoted by $X.$ We identify $\mathcal{H}$ with its
dual. The scalar product is denoted by $((,))$ and the norm by ||.||.
We then consider functionals $\mathcal{F}(X)$ and $\mathcal{F}_{T}(X)$
which are continuously differentiable on $\mathcal{H}$. The gradients
$D_{X}\mathcal{F}(X)$ and $D_{X}\mathcal{F}_{T}(X)$ are Lipschitz
continuous 

\begin{align}
||D_{X}\mathcal{F}(X_{1})-D_{X}\mathcal{F}(X_{2})|| & \leq c||X_{1}-X_{2}||\label{eq:2-1}\\
||D_{X}\mathcal{F}_{T}(X_{1})-D_{X}\mathcal{F}_{T}(X_{2})|| & \leq c||X_{1}-X_{2}||\nonumber 
\end{align}
To simplify notation, we shall also assume that 

\begin{equation}
||D_{X}\mathcal{F}(0)||,\:||D_{X}\mathcal{F}_{T}(0)||\leq c\label{eq:2-1-1}
\end{equation}
So we have 

\begin{equation}
||D_{X}\mathcal{F}(X)||\leq c(1+||X||)\label{eq:2-11-1}
\end{equation}
and 

\begin{equation}
|\mathcal{F}(X)|\leq C(1+||X||^{2}),\;\label{eq:2-12-4}
\end{equation}
where we denote by $C$ a generic constant. The same estimates hold
also for $\mathcal{F}_{T}(X).$ 

A control is a function $v(s)$ which belongs to $L^{2}(0,T;\mathcal{H}).$
We associate to a control $v(.)$ the state $X(s)$ satisfying 

\begin{align}
\frac{dX}{ds} & =v(s)\label{eq:2-2}\\
X(t) & =X\nonumber 
\end{align}
We may write it as $X_{Xt}(s)$ to emphasize the initial conditions
and even $X_{Xt}(s;v(.))$ to emphasize the dependence in the control.
The function $X(s)$ belongs to the Sobolev space $H^{1}(t,T;\mathcal{H}).$
We then define the cost functional 

\begin{equation}
J_{Xt}(v(.))=\frac{\lambda}{2}\int_{t}^{T}||v(s)||^{2}ds+\int_{t}^{T}\mathcal{F}(X(s))ds+\mathcal{F}_{T}(X(T))\label{eq:2-3}
\end{equation}
and the value function 

\begin{equation}
V(X,t)=\inf_{v(.)}\,J_{Xt}(v(.))\label{eq:2-4}
\end{equation}

\subsection{BELLMAN EQUATION}

We want to show the following 
\begin{thm}
\label{theo2-1} We assume (\ref{eq:2-1}), (\ref{eq:2-1-1}) and 

\begin{equation}
\lambda>cT(1+T)\label{eq:2-5}
\end{equation}
The value function (\ref{eq:2-4}) is $C^{1}$ and satisfies the growth
conditions 

\begin{align}
|V(X,t)| & \leq C(1+||X||^{2})\label{eq:2-6}\\
||D_{X}V(X,t)|| & \leq C(1+||X||),\:|\frac{\partial V(X,t)}{\partial t}|\leq C(1+||X||^{2})\nonumber 
\end{align}
where $C$ is a generic constant. Moreover $D_{X}V(X,t)$ and $\dfrac{\partial V(X,t)}{\partial t}$
are Lipschitz continuous, more precisely 

\begin{align}
||D_{X}V(X^{1},t^{1})-D_{X}V(X^{2},t^{2})|| & \leq C||X^{1}-X^{2}||+C|t^{1}-t^{2}|(1+||X^{1}||+||X^{2}||)\label{eq:2-7}\\
|\dfrac{\partial V(X^{1},t^{1})}{\partial t}-\dfrac{\partial V(X^{2},t^{2})}{\partial t}| & \leq C||X^{1}-X^{2}||(1+||X^{1}||+||X^{2}||)+C|t^{1}-t^{2}|(1+||X^{1}||^{2}+||X^{2}||^{2})\nonumber 
\end{align}

It is the unique solution, satisfying conditions (\ref{eq:2-6}) and
(\ref{eq:2-7}) of Bellman equation 
\end{thm}
\begin{align}
\dfrac{\partial V}{\partial t}-\frac{1}{2\lambda}||D_{X}V||^{2}+\mathcal{F}(X) & =0\nonumber \\
V(X,T)=\mathcal{F}_{T}(X)\label{eq:2-8}
\end{align}

The control problem (\ref{eq:2-2}), (\ref{eq:2-3}) has a unique
solution. 
\begin{proof}
We begin by studying the properties of the cost functional $J_{Xt}(v(.)).$
We first claim that $J_{Xt}(v(.))$ is Gâteaux differentiable in the
space $L^{2}(t,T;\mathcal{H})$, for $X,t$ fixed. Define $X_{v}(s)$
by 

\[
\frac{dX_{v}(s)}{ds}=v(s),\:X_{v}(t)=X
\]
 and $Z_{v}(s)$ by 

\[
-\frac{dZ_{v}(s)}{ds}=D_{X}\mathcal{F}(X_{v}(s)),\:Z_{v}(t)=D_{X}\mathcal{F}_{T}(X_{v}(T))
\]
then we can prove easily that 

\begin{equation}
\frac{d}{d\mu}J_{Xt}(v(.)+\mu\tilde{v}(.))|_{\mu=0}=\int_{t}^{T}((\lambda v(s)+Z_{v}(s),\tilde{v}(s)))ds\label{eq:2-9}
\end{equation}
Let us prove that the functional $J_{Xt}(v(.))$ is strictly convex.
Let $v_{1}(.)$ and $v_{2}(.)$ in $L^{2}(t,T;\mathcal{H}).$ We write 

\[
J_{Xt}(\theta v_{1}(.)+(1-\theta)v_{2}(.))=J_{Xt}(v_{1}(.)+(1-\theta)(v_{2}(.)-v_{1}(.)))
\]
\[
=J_{Xt}(v_{1}(.))+\int_{0}^{1}\frac{d}{d\mu}J_{Xt}(v_{1}(.)+\mu(1-\theta)(v_{2}(.)-v_{1}(.)))\,d\mu
\]
 From formula (\ref{eq:2-9}) we have also 

\[
\frac{d}{d\mu}J_{Xt}(v(.)+\mu\theta\tilde{v}(.))=\theta\int_{t}^{T}((\lambda(v(s)+\mu\theta\tilde{v}(.))+Z_{v+\mu\theta\tilde{v}}(s),\tilde{v}(s)))ds
\]
Therefore

\[
\int_{0}^{1}\frac{d}{d\mu}J_{Xt}(v_{1}(.)+\mu(1-\theta)(v_{2}(.)-v_{1}(.)))\,d\mu=(1-\theta)\int_{0}^{1}d\mu\left\{ \int_{t}^{T}((\lambda(v_{1}(s)+\mu(1-\theta)(v_{2}(s)-v_{1}(s)))\right.
\]
\[
+\left.Z_{v_{1}+\mu(1-\theta)(v_{2}-v_{1})}(s),\,v_{2}(s)-v_{1}(s)\,))ds\right\} 
\]
Similarly we write 

\[
J_{Xt}(\theta v_{1}(.)+(1-\theta)v_{2}(.))=J_{Xt}(v_{2}(.)+\theta(v_{1}(.)-v_{2}(.)))
\]
\[
=J_{Xt}(v_{2}(.))+\int_{0}^{1}\frac{d}{d\mu}J_{Xt}(v_{2}(.)+\mu\theta(v_{1}(.)-v_{2}(.)))\,d\mu
\]
 and 

\begin{align*}
\int_{0}^{1}\frac{d}{d\mu}J_{Xt}(v_{2}(.)+\mu\theta(v_{1}(.)-v_{2}(.)))\,d\mu & =\theta\int_{0}^{1}d\mu\left\{ \int_{t}^{T}((\lambda(v_{2}(s)+\mu\theta(v_{1}(s)-v_{2}(s)))\right.+\\
+ & \left.Z_{v_{2}+\mu\theta(v_{1}-v_{2})}(s),\,v_{1}(s)-v_{2}(s)\,))ds\right\} 
\end{align*}
 We shall set $Z_{1}(s)=Z_{v_{1}+\mu(1-\theta)(v_{2}-v_{1})}(s)$
and $Z_{2}(s)=Z_{v_{2}+\mu\theta(v_{1}-v_{2})}(s).$ Combining formulas,
we can write 

\begin{equation}
J_{Xt}(\theta v_{1}(.)+(1-\theta)v_{2}(.))=\theta J_{Xt}(v_{1}(.))+(1-\theta)J_{Xt}(v_{2}(.))+\label{eq:2-10}
\end{equation}
\[
+\theta(1-\theta)\left[-\frac{\lambda}{2}\int_{t}^{T}||v_{1}(s)-v_{2}(s)||^{2}ds+\int_{0}^{1}d\mu\int_{t}^{T}((Z_{1}(s)-Z_{2}(s),v_{2}(s)-v_{1}(s)\,))ds\right]
\]
Let $X_{1}(s)$ and $X_{2}(s)$ denote the states corresponding to
the controls $v_{1}(.)+\mu(1-\theta)(v_{2}(.)-v_{1}(.))$ and $v_{2}(.)+\mu\theta(v_{1}(.)-v_{2}(.))$.
One checks easily that 

\[
X_{1}(s)-X_{2}(s)=(1-\mu)\int_{t}^{s}(v_{1}(\sigma)-v_{2}(\sigma))d\sigma
\]
 and from the definition of $Z_{1}(.)$, $Z_{2}(.)$ we obtain 
\[
||Z_{1}(s)-Z_{2}(s)||\leq c[||X_{1}(T)-X_{2}(T)||+\int_{s}^{T}||X_{1}(\sigma)-X_{2}(\sigma)||d\sigma]
\]
 and combining formulas, we can assert 

\[
||Z_{1}(s)-Z_{2}(s)||\leq c(1-\mu)(1+T)\int_{t}^{T}||v_{1}(\sigma)-v_{2}(\sigma)||d\sigma
\]
 Going back to (\ref{eq:2-10}) we obtain easily 

\begin{equation}
J_{Xt}(\theta v_{1}(.)+(1-\theta)v_{2}(.))\leq\theta J_{Xt}(v_{1}(.))+(1-\theta)J_{Xt}(v_{2}(.))+\label{eq:2-10-1}
\end{equation}
\[
-\frac{\theta(1-\theta)}{2}(\lambda-cT(1+T))\int_{t}^{T}||v_{1}(s)-v_{2}(s)||^{2}ds
\]
 and from the assumption (\ref{eq:2-5}) we obtain immediately that
$J_{Xt}(v(.))$ is strictly convex. Next we write 

\[
\mathcal{F}(X(s))-\mathcal{F}(X)=\int_{0}^{1}((D_{X}\mathcal{F}(X+\theta\int_{t}^{s}v(\sigma)d\sigma),\int_{t}^{s}v(\sigma)d\sigma))
\]
 so , using (\ref{eq:2-11-1}) we obtain 

\[
|\mathcal{F}(X(s))-\mathcal{F}(X)|\leq c(1+||X||)||\int_{t}^{s}v(\sigma)d\sigma||+\frac{c}{2}||\int_{t}^{s}v(\sigma)d\sigma||^{2}
\]
\[
\leq\frac{c^{2}(1+||X||)^{2}}{2\delta}+\frac{c+\delta}{2}||\int_{t}^{s}v(\sigma)d\sigma||^{2}
\]
for any $\delta>0.$ Using (\ref{eq:2-12-4}) we can assert that 

\[
|\mathcal{F}(X(s))|\leq C_{\delta}(1+||X||^{2})+\frac{c+\delta}{2}T\int_{t}^{T}||v(\sigma)||^{2}d\sigma
\]
A similar estimate holds for $\mathcal{F}_{T}(X(T)).$ Therefore,
collecting results, we obtain
\begin{align*}
|\int_{t}^{T}\mathcal{F}(X(s))ds+\mathcal{F}_{T}(X(T))| & \leq C_{\delta}(1+||X||^{2})(1+T)\\
+ & \frac{c+\delta}{2}T(1+T)\int_{t}^{T}||v(s)||^{2}ds
\end{align*}
It follows that 

\begin{align}
J_{Xt}(v(.)) & \geq\frac{\lambda-(c+\delta)T(1+T)}{2}\int_{t}^{T}||v(s)||^{2}ds-C_{\delta}(1+||X||^{2})(1+T)\label{eq:2-12-2}
\end{align}
Since $\lambda-cT(1+T)>0,$we can find $\delta>0$ sufficiently small
so that $\lambda-(c+\delta)T(1+T)>0.$ This implies that $J_{Xt}(v(.))\rightarrow+\infty$
as $\int_{t}^{T}||v(s)||^{2}ds\rightarrow+\infty.$ This property
and the strict convexity imply that the functional $J_{Xt}(v(.))$
has a minimum which is unique. The Gâteaux derivative must vanish
at this minimum denoted by $u(.).$ The corresponding state is denoted
by $Y(.).$ From formula (\ref{eq:2-9}) we obtain also the existence
of a solution of the two-point boundary value problem 

\begin{align}
\frac{dY}{ds} & =-\frac{Z(s)}{\lambda},\quad-\frac{dZ}{ds}=D_{X}\mathcal{F}(Y(s))\label{eq:2-11}\\
Y(t) & =X,\qquad Z(T)=D_{X}\mathcal{F}_{T}(Y(T))\nonumber 
\end{align}
and the optimal control $u(.)$ is given by the formula 

\begin{equation}
u(s)=-\frac{Z(s)}{\lambda}\label{eq:2-12}
\end{equation}
In fact, the system (\ref{eq:2-11}) can be studied directly, and
we can show directly that it has one and only one solution. We notice
that it is a 2nd order differential equation, since 

\begin{align}
\frac{d^{2}Y}{ds^{2}} & =\frac{1}{\lambda}D_{X}\mathcal{F}(Y(s))\label{eq:2-12-1}\\
Y(t)=X & \quad\frac{dY}{ds}(T)=-\frac{1}{\lambda}D_{X}\mathcal{F}_{T}(Y(T))\nonumber 
\end{align}
 We can write also (\ref{eq:2-12-1}) as an integral equation 

\begin{equation}
Y(s)=X-\frac{s-t}{\lambda}D_{X}\mathcal{F}_{T}(Y(T))-\frac{1}{\lambda}\int_{t}^{T}D_{X}\mathcal{F}(Y(\sigma))(s\wedge\sigma-t)d\sigma\label{eq:2-13}
\end{equation}
and we can view this equation as a fixed point equation in the space
$C^{0}([t,T];\mathcal{H}),$ namely $Y(.)=\mathcal{K}(Y(.)),$ where
$\mathcal{K}$ is defined by the right hand side of (\ref{eq:2-13}).
One can show that $\mathcal{K}$ is a contraction, hence $Y(.)$ is
uniquely defined. Note also, that if we have a solution of (\ref{eq:2-11})
and if $u(.)$ is defined by (\ref{eq:2-12}) the control $u(.)$
satisfies the necessary condition of optimality for the functional
$J_{Xt}(v(.)).$ Since this functional is convex, the necessary condition
of optimality is also sufficient and thus $u(.)$ is optimal. The
value function is thus defined by the formula 

\begin{equation}
V(X,t)=\frac{1}{2\lambda}\int_{t}^{T}||Z(s)||^{2}ds+\int_{t}^{T}\mathcal{F}(Y(s))ds+\mathcal{F}_{T}(Y(T))\label{eq:2-14}
\end{equation}
We now study the properties of the value function. We begin with the
first property (\ref{eq:2-6}). Using (\ref{eq:2-12-2}) we obtain 

\[
V(X,t)\geq-C(1+||X||^{2})
\]
 On the other hand, we have 

\[
V(X,t)\leq J_{Xt}(0)=(T-t)\mathcal{F}(X)+\mathcal{F}_{T}(X)\leq C(1+||X||^{2})
\]
 and the first estimate (\ref{eq:2-6}) is obtained. 

We proceed in getting estimates for the solution $Y(.)$ of (\ref{eq:2-13}).
We write 

\[
||Y(.)||=\sup_{t\leq s\leq T}||Y(s)||
\]
 Using easy majorations, we obtain 

\begin{align}
||Y(.)|| & \leq\frac{||X||\lambda+cT(T+1)}{\lambda-cT(T+1)}\label{eq:2-15}\\
||Z(.)|| & \leq\frac{\lambda(1+T)c(1+||X||)}{\lambda-cT(T+1)}\nonumber \\
||u(.)|| & \leq\frac{(1+T)c(1+||X||)}{\lambda-cT(T+1)}\nonumber 
\end{align}

We then study how these functions depend on the pair $X,t.$ We recall
that $Y(s)=Y_{Xt}(s).$ Let us consider two points $X_{1},t_{1}$
and $X_{2},t_{2}$ and denote $Y_{1}(s)=Y_{X_{1}t_{1}}(s),$ $Y_{2}(s)=Y_{X_{2}t_{2}}(s)$.
To fix ideas we assume $t_{1}<t_{2}.$ For $s>t_{2}$ we have 

\begin{align*}
Y_{1}(s)-Y_{2}(s) & =X_{1}-X_{2}-\frac{1}{\lambda}(D_{X}\mathcal{F}_{T}(Y_{1}(T))-D_{X}\mathcal{F}_{T}(Y_{2}(T)))(s-t_{2})-\frac{1}{\lambda}D_{X}\mathcal{F}_{T}(Y_{1}(T))(t_{2}-t_{1})\\
-\frac{1}{\lambda} & \int_{t_{2}}^{T}(D_{X}\mathcal{F}(Y_{1}(\sigma))-D_{X}\mathcal{F}(Y_{2}(\sigma)))(s\wedge\sigma-t_{2})d\sigma-\frac{1}{\lambda}\int_{t_{1}}^{t_{2}}D_{X}\mathcal{F}(Y_{1}(\sigma))(s\wedge\sigma-t_{1})d\sigma
\end{align*}
 From which we obtain 
\begin{align*}
\sup_{t_{2}\leq s\leq T}||Y_{1}(s)-Y_{2}(s)|| & \leq||X_{1}-X_{2}||+\frac{c}{\lambda}T(1+T)\sup_{t_{2}\leq s\leq T}||Y_{1}(s)-Y_{2}(s)||+\\
+ & \frac{t_{2}-t_{1}}{\lambda}[||D_{X}\mathcal{F}_{T}(Y_{1}(T))||+\int_{t_{1}}^{T}||D_{X}\mathcal{F}(Y_{1}(s))||ds]
\end{align*}
Using the properties of $D_{X}\mathcal{F}$ and $D_{X}\mathcal{F}_{T}$
and (\ref{eq:2-15}) we can assert that 

\[
\sup_{t_{2}\leq s\leq T}||Y_{1}(s)-Y_{2}(s)||\leq\frac{\lambda}{\lambda-cT(T+1)}\left(||X_{1}-X_{2}||+(t_{2}-t_{1})(1+T)c\frac{1+||X_{1}||}{\lambda-cT(T+1)}\right)
\]
 More globally we can write 
\begin{equation}
\sup_{\max(t_{1},t_{2})\leq s\leq T}||Y_{X_{1}t_{1}}(s)-Y_{X_{2}t_{2}}(s)||\leq\frac{\lambda}{\lambda-cT(T+1)}\left(||X_{1}-X_{2}||+|t_{2}-t_{1}|(1+T)c\frac{1+\,\max(||X_{1}||,||X_{2}||)}{\lambda-cT(T+1)}\right)\label{eq:2-16}
\end{equation}

In particular 

\begin{equation}
\sup_{t\leq s\leq T}||Y_{X_{1}t}(s)-Y_{X_{2}t}(s)||\leq\frac{\lambda||X_{1}-X_{2}||}{\lambda-cT(T+1)}\label{eq:2-17}
\end{equation}
Recalling that from the system (\ref{eq:2-11}) we have 

\[
Z(s)=\int_{s}^{T}D_{X}\mathcal{F}(Y(\sigma))d\sigma+D_{X}\mathcal{F}_{T}(Y(T))
\]
 and noting $Z(s)=Z_{Xt}(s)$ we deduce from (\ref{eq:2-17}) that 

\begin{equation}
\sup_{t\leq s\leq T}||Z_{X_{1}t}(s)-Z_{X_{2}t}(s)||\leq\frac{c(T+1)\lambda||X_{1}-X_{2}||}{\lambda-cT(T+1)}\label{eq:2-18}
\end{equation}
We next write 

\[
J_{X_{1}t}(u_{1}(.))-J_{X_{2}t}(u_{1}(.))\leq V(X_{1},t)-V(X_{2},t)\leq J_{X_{1}t}(u_{2}(.))-J_{X_{2}t}(u_{2}(.))
\]
where $u_{1}(.)$ and $u_{2}(.)$ are the optimal controls for the
problems with initial conditions $(X_{1},t)$ and $(X_{2},t),$ respectively.
Denoting by $Y_{X_{1}t}(s)$ and $Y_{X_{2}t}(s)$ the optimal states
and by $Y_{X_{1}t}(s;u_{2}(.)),\;Y_{X_{2}t}(s;u_{1}(.))$ the trajectories
( not optimal) when the control $u_{2}(.)$ is used with the initial
conditions $(X_{1},t)$ and when the control $u_{1}(.)$ is used with
the initial conditions $(X_{2},t),$ we have 

\begin{align*}
Y_{X_{1}t}(s;u_{2}(.))-Y_{X_{2}t}(s) & =Y_{X_{1}t}(s)-Y_{X_{2}t}(s;u_{1}(.))=X_{1}-X_{2}
\end{align*}
Therefore 
\[
V(X_{1},t)-V(X_{2},t)\leq\int_{t}^{T}(\mathcal{F}(Y_{X_{2}t}(s)+X_{1}-X_{2})-\mathcal{F}(Y_{X_{2}t}(s)))ds+\mathcal{F}_{T}(Y_{X_{2}t}(T)+X_{1}-X_{2})-\mathcal{F}_{T}(Y_{X_{2}t}(T))
\]
 and by techniques already used it follows 
\[
V(X_{1},t)-V(X_{2},t)\leq((\int_{t}^{T}D_{X}\mathcal{F}(Y_{X_{2}t}(s))ds+D_{X}\mathcal{F}_{T}(Y_{X_{2}t}(T)),X_{1}-X_{2}))+\frac{c}{2}(1+T)||X_{1}-X_{2}||^{2}
\]
 which is in fact 

\begin{equation}
V(X_{1},t)-V(X_{2},t)\leq((Z_{X_{2}t}(t),X_{1}-X_{2}))+\frac{c}{2}(1+T)||X_{1}-X_{2}||^{2}\label{eq:2-19}
\end{equation}
By interchanging the roles of $X_{1}$ and $X_{2}$ we also obtain 

\begin{equation}
V(X_{1},t)-V(X_{2},t)\geq((Z_{X_{1}t}(t),X_{1}-X_{2}))-\frac{c}{2}(1+T)||X_{1}-X_{2}||^{2}\label{eq:2-20}
\end{equation}
Using the estimate (\ref{eq:2-18}) we can also write 

\begin{equation}
V(X_{1},t)-V(X_{2},t)\geq((Z_{X_{2}t}(t),X_{1}-X_{2}))-c(T+1)[\frac{\lambda}{\lambda-cT(T+1)}+\frac{1}{2}]||X_{1}-X_{2}||^{2}\label{eq:2-21}
\end{equation}
Combining (\ref{eq:2-19}) and (\ref{eq:2-21}) we immediately get 

\begin{equation}
|V(X_{1},t)-V(X_{2},t)-((Z_{X_{2}t}(t),X_{1}-X_{2}))|\leq c(T+1)[\frac{\lambda}{\lambda-cT(T+1)}+\frac{1}{2}]||X_{1}-X_{2}||^{2}\label{eq:2-22}
\end{equation}
This shows immediately that $V(X,t)$ is differentiable in $X$ and
that 

\begin{equation}
D_{X}V(X,t)=Z(t)=-\lambda u(t)\label{eq:2-23}
\end{equation}
From the 2nd estimate (\ref{eq:2-15}) we immediately obtain the 2nd
estimate (\ref{eq:2-6}). We continue with the derivative in $t.$
We first write the optimality principle 

\begin{equation}
V(X,t)=\frac{\lambda}{2}\int_{t}^{t+\epsilon}||u(s)||^{2}ds+\int_{t}^{t+\epsilon}\mathcal{F}(Y(s))ds+V(Y(t+\epsilon),t+\epsilon)\label{eq:2-24}
\end{equation}
which is a simple consequence of the definition of the value function
and of the existence of an optimal control. From (\ref{eq:2-22})
we can write

\[
V(X_{2},t)-V(X_{1},t)-((Z_{X_{2}t}(t),X_{2}-X_{1}))\leq C||X_{1}-X_{2}||^{2}
\]
 where $C$ is the constant appearing in the right hand side of (\ref{eq:2-22}).
We apply with $X_{2}=Y(t+\epsilon)$ , $X_{1}=X,$$t=t+\epsilon.$
We note that $Z_{Y(t+\epsilon),t+\epsilon}(t+\epsilon)=Z_{Xt}(t+\epsilon)=-\lambda u(t+\epsilon),$since
$u(s)$ for $t+\epsilon<s<t$ is optimal for the problem starting
with initial conditions $Y(t+\epsilon),\,t+\epsilon.$ Therefore 

\[
V(Y(t+\epsilon),t+\epsilon)-V(X,t+\epsilon)\leq-\lambda((u(t+\epsilon),\int_{t}^{t+\epsilon}u(s)ds))+C||\int_{t}^{t+\epsilon}u(s)ds||^{2}
\]
Using this inequality in (\ref{eq:2-24}) yields 

\[
V(X,t)-V(X,t+\epsilon)\leq\frac{\lambda}{2}\int_{t}^{t+\epsilon}||u(s)||^{2}ds+\int_{t}^{t+\epsilon}\mathcal{F}(Y(s))ds-\lambda((u(t+\epsilon),\int_{t}^{t+\epsilon}u(s)ds))+C||\int_{t}^{t+\epsilon}u(s)ds||^{2}
\]
from which we obtain 

\begin{equation}
\liminf_{\epsilon\rightarrow0}\frac{V(X,t+\epsilon)-V(X,t)}{\epsilon}\geq\frac{\lambda}{2}||u(t)||^{2}-\mathcal{F}(X)\label{eq:2-25}
\end{equation}
Next we have 

\[
V(X,t+\epsilon)\leq\frac{\lambda}{2}\int_{t+\epsilon}^{T}||u(s)||^{2}ds+\int_{t+\epsilon}^{T}\mathcal{F}(Y(s)-\int_{t}^{t+\epsilon}u(\sigma)d\sigma)ds+\mathcal{F}_{T}(Y(T)-\int_{t}^{t+\epsilon}u(\sigma)d\sigma)
\]
 therefore 
\begin{align*}
V(X,t+\epsilon)-V(X,t) & \leq-\frac{\lambda}{2}\int_{t}^{t+\epsilon}||u(s)||^{2}ds-\int_{t}^{t+\epsilon}\mathcal{F}(Y(s))ds+\\
+ & \int_{t+\epsilon}^{T}(\mathcal{F}(Y(s)-\int_{t}^{t+\epsilon}u(\sigma)d\sigma)-\mathcal{F}(Y(s)))ds+\\
+ & \mathcal{F}_{T}(Y(T)-\int_{t}^{t+\epsilon}u(\sigma)d\sigma)-\mathcal{F}_{T}(Y(T))
\end{align*}
 and using assumptions on $\mathcal{F}$, $\mathcal{F}_{T}$ it follows
that 
\[
V(X,t+\epsilon)-V(X,t)\leq-\frac{\lambda}{2}\int_{t}^{t+\epsilon}||u(s)||^{2}ds-\int_{t}^{t+\epsilon}\mathcal{F}(Y(s))ds+
\]
\[
-((\int_{t+\epsilon}^{T}D_{X}\mathcal{F}(Y(s))ds+D_{X}\mathcal{F}_{T}(Y(T)),\int_{t}^{t+\epsilon}u(\sigma)d\sigma))+\frac{c}{2}(1+T)||\int_{t}^{t+\epsilon}u(\sigma)d\sigma||^{2}
\]
which means 
\begin{align*}
V(X,t+\epsilon)-V(X,t) & \leq-\frac{\lambda}{2}\int_{t}^{t+\epsilon}||u(s)||^{2}ds-\int_{t}^{t+\epsilon}\mathcal{F}(Y(s))ds+\\
+ & \lambda((u(t+\epsilon),\int_{t}^{t+\epsilon}u(\sigma)d\sigma))+\frac{c}{2}(1+T)||\int_{t}^{t+\epsilon}u(\sigma)d\sigma||^{2}
\end{align*}
 We then obtain 

\begin{equation}
\limsup_{\epsilon\rightarrow0}\frac{V(X,t+\epsilon)-V(X,t)}{\epsilon}\leq\frac{\lambda}{2}||u(t)||^{2}-\mathcal{F}(X)\label{eq:2-26}
\end{equation}
and comparing with (\ref{eq:2-25}) we obtain immediately that $V(X,t)$
is differentiable in $t$ and the derivative is given by 

\begin{equation}
\frac{\partial V}{\partial t}(X,t)=\frac{\lambda}{2}||u(t)||^{2}-\mathcal{F}(X)\label{eq:2-27}
\end{equation}
Recalling (\ref{eq:2-23}) we see immediately that $V(X,t)$ is solution
of the HJB equation (\ref{eq:2-8}). The 2nd estimate (\ref{eq:2-6})
is an immediate consequence of the equation and the estimate on $||D_{X}V(X,t)||.$
We next turn to check the addtional estimates (\ref{eq:2-7}). We
have 

\[
D_{X}V(X_{1},t_{1})-D_{X}V(X_{2},t_{2})=Z_{X_{1}t_{1}}(t_{1})-Z_{X_{2}t_{2}}(t_{2})
\]
We assume $t_{1}<t_{2}$ then we can write 

\begin{align}
Z_{X_{1}t_{1}}(t_{1})-Z_{X_{2}t_{2}}(t_{2}) & =\int_{t_{1}}^{t_{2}}(D_{X}\mathcal{F}(Y_{X_{1}t_{1}}(s))-D_{X}\mathcal{F}(Y_{X_{2}t_{2}}(s)))ds+\nonumber \\
+ & D_{X}\mathcal{F}_{T}(Y_{X_{1}t_{1}}(T))-D_{X}\mathcal{F}_{T}(Y_{X_{2}t_{2}}(T))\label{eq:2-28}
\end{align}
Using previously used majorations, we can check 

\begin{equation}
||Z_{X_{1}t_{1}}(t_{1})-Z_{X_{2}t_{2}}(t_{2})||\leq\frac{\lambda c(T+1)}{\lambda-cT(T+1)}\left(||X_{1}-X_{2}||+|t_{2}-t_{1}|(1+T)c\frac{1+\,\max(||X_{1}||,||X_{2}||)}{\lambda-cT(T+1)}\right)\label{eq:2-29}
\end{equation}
and the first estimate (\ref{eq:2-7}) follows immediately. The 2nd
estimate (\ref{eq:2-7}) is a direct consequence of the HJB equation
and of the first estimate (\ref{eq:2-7}). So the value function has
the regularity indicated in the statement and satisfies the HJB equation.
Let us show that such a solution is necessarily unique. This is a
consequence of the verification property. Indeed consider any control
$v(.)\in L^{2}(t,T;\mathcal{H})$ and the state $X(s)$ solution of
(\ref{eq:2-2}). Let $V(x,t)$ be a solution of the HJB equation which
is $C^{1}$ and satisifies (\ref{eq:2-6}), (\ref{eq:2-7}). Then
the function $V(X(s),s)$ is differentiable and 

\begin{align*}
\frac{d}{ds}V(X(s),s) & =\frac{\partial V}{\partial s}(X(s),s)+((D_{X}V(X(s),s),v(s)))=\\
= & -\mathcal{F}(X(s))+\frac{1}{2\lambda}||D_{X}V(X(s),s)||^{2}+((D_{X}V(X(s),s),v(s)))\\
\geq & -\mathcal{F}(X(s))-\frac{\lambda}{2}||v(s)||^{2}
\end{align*}
 from which we get immediately by inegration $V(X,t)\leq J_{Xt}(v(.)).$
Now if we consider the equation 

\begin{equation}
\dfrac{d\hat{X}(s)}{ds}=-\frac{1}{\lambda}D_{X}V(\hat{X}(s),s),\;\hat{X}(t)=X\label{eq:2-30}
\end{equation}
it has a unique solution, since $D_{X}V(X,s)$ is uniformly Lipschitz
in $X.$ If we set $\hat{v}(s)=-\dfrac{1}{\lambda}D_{X}V(\hat{X}(s),s),$
we see easily that $V(X,t)=J_{Xt}(\hat{v}(.))$. So $V(X,t)$ coincides
with the value function, and thus we have only one possible solution.
This completes the proof of the theorem. $\blacksquare$
\end{proof}

\section{THE MASTER EQUATION }

\subsection{FURTHER REGULARITY ASSUMPTIONS. }

We now assume that 

\begin{equation}
\mathcal{F},\:\mathcal{F}_{T}\,\text{are}\:C^{2}\label{eq:3-1}
\end{equation}
The operators $D_{X}^{2}\mathcal{F}(X)$,$D_{X}^{2}\mathcal{F}_{T}(X)$
belong to $\mathcal{L}(\mathcal{H};\mathcal{H}).$ According to the
assumptions (\ref{eq:2-1}) we can assert that 

\begin{equation}
||D_{X}^{2}\mathcal{F}(X)||,\:||D_{X}^{2}\mathcal{F}_{T}(X)||\leq c\label{eq:3-2}
\end{equation}
where the norm of the operators is the norm of $\mathcal{L}(\mathcal{H};\mathcal{H}).$
Recalling the equation (\ref{eq:2-13}) for $Y(s),$ we differentiate
formally with respect to $X$ to obtain 

\begin{align}
D_{X}Y(s) & =I-\frac{s-t}{\lambda}D_{X}^{2}\mathcal{F}_{T}(Y(T))D_{X}Y(T)\label{eq:3-3}\\
-\frac{1}{\lambda} & \int_{t}^{T}D_{X}^{2}\mathcal{F}(Y(\sigma))D_{X}Y(\sigma)(s\wedge\sigma-t)d\sigma\nonumber 
\end{align}
so, $D_{X}Y(.)$ appears as the solution of a linear equation, and
we see easily that it has one and only one solution verifying 

\begin{equation}
\sup_{t\leq s\leq T}||D_{X}Y(s)||\leq\frac{\lambda}{\lambda-cT(T+1)}\label{eq:3-4}
\end{equation}
It is then easy to check that $D_{X}Y(s)$ is indeed the gradient
of $Y_{Xt}(s)$ with respect to $X$, and the estimate (\ref{eq:3-4})
is coherent with (\ref{eq:2-7}). Since $D_{X}V(X,t)=Z(t)=Z_{Xt}(t)$
with 

\[
Z(t)=\int_{t}^{T}D_{X}\mathcal{F}(Y(s))ds+D_{X}\mathcal{F}_{T}(Y(T))
\]
 we can differentiate to obtain 

\begin{equation}
D_{X}^{2}V(X,t)=\int_{t}^{T}D_{X}^{2}\mathcal{F}(Y(s))D_{X}Y(s)ds+D_{X}^{2}\mathcal{F}_{T}(Y(T))D_{X}Y(T)\label{eq:3-5}
\end{equation}
and 

\begin{equation}
||D_{X}^{2}V(X,t)||\leq\frac{\lambda c(T+1)}{\lambda-cT(T+1)}\label{eq:3-6}
\end{equation}
which is coherent with (\ref{eq:2-18}).

\subsection{MASTER EQUATION}

We obtain the Master equation, by simply differentiating the HJB equation
(\ref{eq:2-8}) with respect to $X.$ We set $\mathcal{U}(X,t)=D_{X}V(X,t)$.
We know from (\ref{eq:2-15}) that 

\begin{equation}
||\mathcal{U}(X,t)||\leq\frac{\lambda(1+T)c(1+||X||)}{\lambda-cT(T+1)}\label{eq:3-7}
\end{equation}
The function $\mathcal{U}(X,t)$ maps $\mathcal{H}\times(0,T)$ into
$\mathcal{H}.$ From (\ref{eq:3-6}) we see that it is differentiable
in $X,$ with $D_{X}\mathcal{U}(X,t):\mathcal{H}\times(0,T)\rightarrow\mathcal{L}(\mathcal{H};\mathcal{H})$
and 

\begin{equation}
||D_{X}\mathcal{U}(X,t)||\leq\frac{\lambda c(1+T)}{\lambda-cT(T+1)}\label{eq:3-8}
\end{equation}
 From the HJB equation we see that $\mathcal{U}(X,t)$ is differentiable
in $t$ and satisfies the equation 

\begin{align}
\frac{\partial\mathcal{U}}{\partial t}-\frac{1}{\lambda}D_{X}\mathcal{U}(X,t)\,\mathcal{U}(X,t)+D_{X}\mathcal{F}(X) & =0\label{eq:3-9}\\
\mathcal{U}(X,T)=D_{X}\mathcal{F}_{T}(X)\nonumber 
\end{align}
 We have the 
\begin{prop}
\label{prop2} We make the assumptions of Theorem\ref{theo2-1} and
(\ref{eq:3-1}). Then equation (\ref{eq:3-9}) has one and only one
solution satisfying the estimates (\ref{eq:3-7}), (\ref{eq:3-8}).\end{prop}
\begin{proof}
We have only to prove uniqueness. Noting that 

\[
D_{X}\mathcal{U}(X,t)\,\mathcal{U}(X,t)=\frac{1}{2}D_{X}||\mathcal{U}(X,t)||^{2}
\]
 we see immediately from the equation that $\mathcal{U}(X,t)$ is
a gradient. So $\mathcal{U}(X,t)=D_{X}\tilde{V}(X,t).$ Therefore
(\ref{eq:3-9}) reads 

\begin{align*}
D_{X}(\frac{\partial\tilde{V}}{\partial t}-\frac{1}{2\lambda}||D_{X}\tilde{V}||^{2}+\mathcal{F}(X)) & =0\\
D_{X}\tilde{V}(X,T)=D_{X}\mathcal{F}_{T}(X)
\end{align*}
 We thus can write 

\begin{align*}
\frac{\partial\tilde{V}}{\partial t}-\frac{1}{2\lambda}||D_{X}\tilde{V}||^{2}+\mathcal{F}(X) & =f(t)\\
\tilde{V}(X,T)=D_{X}\mathcal{F}_{T}(X)+h
\end{align*}
where $f(t)$ is purely function of $t$ and $h$ is a constant. If
we introduce the function $\varphi(t)$ solution of 

\[
\frac{\partial\varphi}{\partial t}=f(t),\quad\varphi(T)=h
\]
 the function $\tilde{V(}X,t)-\varphi(t)$ is solution of the HJB
equation (\ref{eq:2-8}) and satisfies the regularity properties of
Theorem \ref{theo2-1}. From the uniqueness of the solution of the
HJB equation we have $\tilde{V(}X,t)-\varphi(t)=V(X,t)$ the value
function, hence $\mathcal{U}(X,t)==D_{X}V(X,t),$ which proves the
uniqueness.$\blacksquare$
\end{proof}

\section{FUNCTIONALS ON PROBABILITY MEASURES}

\subsection{GENERAL COMMENTS}

If we have a functional on probability measures, the idea , introduced
by P.L. Lions \cite{PLL}, \cite{PL2} is to consider it as a functional
on random variables, whose probability laws are the probability measures.
Nevertheless, it is possible to work with the space of probability
measures directly, which is a metric space. The key issue is to define
the concept of gradient. For the space of probability measures, it
is the Wasserstein gradient. We shall see that, in fact, it is equivalent
to the gradient in the sence of the Hilbert space of random variables.

\subsection{WASSERSTEIN GRADIENT}

We consider the space $\mathcal{P}_{2}(R^{n})$ of probability measures
on $R^{n}$ , with second order moments, equipped with the Wasserstein
metric $W_{2}(\mu,\nu),$defined by 

\begin{equation}
W_{2}^{2}(\mu,\nu)=\inf_{\gamma\in\Gamma(\mu,\nu)}\int_{R^{n}\times R^{n}}|\xi-\eta|^{2}\gamma(d\xi,d\eta)\label{eq:4-1}
\end{equation}
where $\Gamma(\mu,\nu)$ denotes the set of joint probability measures
on $R^{n}\times R^{n}$ such that the marginals are $\mu$ and $\nu$
respectively. It is useful to consider a probability space $\Omega,$$\mathcal{A},P$
and random variables in $\mathcal{H=}$$L^{2}(\Omega,\mathcal{A},P;R^{n}).$
We then can write $\mu=\mathcal{L}_{X}$ and 

\[
W_{2}^{2}(\mu,\nu)=\inf_{\begin{array}{c}
X,Y\in\mathcal{H}\\
\mathcal{L}_{X}=\mu\\
\mathcal{L}_{Y}=\nu
\end{array}}E|X-Y|^{2}
\]
 When the probability law has a density with respect to Lebesgue measure,
say $m(x)$ belonging to $L^{1}(R^{n})$ and positive, we replace
the law by its density. Note that $\int|x|^{2}m(x)dx$ $<+\infty.$
We call $L_{m}^{2}(R^{n};R^{n})$ the space of functions $f:\,R^{n}\rightarrow R^{n}$
such that $\int_{R^{n}}|f(x)|^{2}m(x)dx<+\infty.$ We consider functionals
$F(\mu)$ on $\mathcal{P}_{2}(R^{n}).$ If $\mu$ has a density $m$
we write $F(m).$ If $m\in L^{2}(R^{n})$, we say that $F(m)$ has
a Gâteaux differential at $m$, denoted by $\dfrac{\partial F(m)}{\partial m}(x)$
if we have 

\begin{equation}
\lim_{\theta\rightarrow0}\dfrac{F(m+\theta\mu)-F(m)}{\theta}=\int_{R^{n}}\dfrac{\partial F(m)}{\partial m}(x)\mu(x)dx,\,\forall\mu\in L^{2}(R^{n})\label{eq:4-2}
\end{equation}
and $\dfrac{\partial F(m)}{\partial m}(x)\in L^{2}(R^{n}).$ For probability
densities, we shall extend this concept as follows. We say that $\dfrac{\partial F(m)}{\partial m}(x)\in L_{m}^{1}(R^{n})$
is the \emph{functional derivative} of $F$ at $m$ if for any sequence
of probability densities $m_{\epsilon}$ in $\mathcal{P}_{2}(R^{n})$
such that $W_{2}(m_{\epsilon},m)\rightarrow0$ then $\dfrac{\partial F(m)}{\partial m}(.)\in$$L_{m_{\epsilon}}^{1}(R^{n})$
and 

\begin{equation}
\frac{F(m_{\epsilon})-F(m)-\int_{R^{n}}\dfrac{\partial F(m)}{\partial m}(x)(m_{\epsilon}(x)-m(x))dx}{W_{2}(m_{\epsilon},m)}\rightarrow0,\:\text{as}\:\epsilon\rightarrow0\label{eq:4-22}
\end{equation}

The function $\dfrac{\partial F(m)}{\partial m}(.)$ is called the
\emph{functional derivative }of $F(m)$ at point $m.$ Let us see
the connection with the concept of Wasserstein gradient on the metric
space $\mathcal{P}_{2}(R^{n})$ . We shall simply give the definition
and the expression of the gradient. For a detailed theory, we refer
to Otto \cite{OTT}, Ambrosio- Gigli- Savaré \cite{AGS}, Benamou-Brenier
\cite{BEB}, Brenier \cite{BRE}, Jordan-Kinderlehrer-Otto \cite{JKO},
Otto \cite{OTT}, Villani \cite{VIL}. 

The first concept is that of optimal transport map, also called Brenier's
map. Given a probability $\nu\:\in\mathcal{P}_{2}(R^{n}),$the Monge
problem 

\[
\inf_{\{T(.)|\,T(.)m=\nu\}}\int_{R^{n}}|x-T(x)|^{2}m(x)dx
\]
 has a unique solution which is a gradient $T(x)=D\Phi(x).$ The notation
$T(.)m=\nu$ means that $\nu$ is the image of the probability whose
density is $m.$ The optimal solution is the Brenier's map. It is
noted $T_{m}^{\nu}$. We do not necessarily assume that $\nu$ has
a density. The following property holds 

\begin{equation}
W_{2}^{2}(m,\nu)=\int_{R^{n}}|x-D\Phi(x)|^{2}m(x)dx\label{eq:4-3}
\end{equation}
This motivates the definition of tangent space $\mathcal{T}(m)$ of
the metric space $\mathcal{P}_{2}(R^{n})$ at point $m$ as 

\[
\mathcal{T}(m)=\overline{\{D\Phi|\Phi\in C_{c}^{\infty}(R^{n})\}}
\]
We next consider curves on $\mathcal{P}_{2}(R^{n}),$ defined by densities
$m(t)\equiv m(t)(x)=m(x,t)$. The evolution of $m(t)$ is defined
by a velocity vector field $v(t)\equiv v(t)(x)=v(x,t)$ if $m(x,t)$
is the solution of the continuity equation 

\begin{align}
\frac{\partial m}{\partial t}+\text{div }(v(x,t)m(x,t)) & =0\label{eq:4-4}\\
m(x,0)=m(x)\nonumber 
\end{align}
We can interpret this equation in the sense of distributions, and
it is sufficient to assume that $\int_{0}^{T}\int_{R^{n}}|v(x,t)|^{2}m(x,t)dxdt<+\infty,\,\forall T<+\infty.$
This evolution model has a broad sprectrum and turns out to be equivalent
to the property that $m(t)$ is absolutely continuous in the sense 

\[
W_{2}(m(s),m(t))\leq\int_{s}^{t}\rho(\sigma)d\sigma,\,\forall s<t
\]
 with $\rho(.)$ locally $L^{2}.$ Now, for a given absolutely continuous
curve $m(t),$ the corresponding velocity field is not necessarily
unique. We can define the velocity field with minimum norm, i.e. $\hat{v}(x,t)$
solution of 

\begin{equation}
\inf\left\{ \int_{0}^{T}\int_{R^{n}}|v(x,t)|^{2}m(x,t)dxdt\,\left|\begin{array}{cc}
\frac{\partial m}{\partial t}+\text{div }(v(x,t)m(x,t)) & =0\end{array}\right.\right\} \label{eq:4-5}
\end{equation}
The Euler equation for this minimization problem is 

\[
\int_{0}^{T}\int_{R^{n}}\hat{v}(x,t).v(x,t)m(x,t)dt=0,\:\forall v(x,t)|\text{div }(v(x,t)m(x,t))=0\,\text{a.e. }
\]
which implies immediately that $\hat{v}(t)\in\mathcal{T}(m(t))$ a.e.
$t.$ Consequently, to a given absolutely continuous curve $m(t)$
we can associate a unique velocity field $\hat{v}(t)$ in the tangent
space $\mathcal{T}(m(t))$ a.e. $t.$ It is called the \uline{tangent
vector field }to the curve $m(t).$ It can be expressed by the following
formula

\begin{equation}
\hat{v}(x,t)=\lim_{\epsilon\rightarrow0}\frac{T_{m(t)}^{m(t+\epsilon)}(x)-x}{\epsilon}\label{eq:4-6}
\end{equation}
the limit being understood in $L_{m(t)}^{2}(R^{n};R^{n}).$ The function
$T_{m(t)}^{m(t+\epsilon)}(x)$ is uniquely defined. Since by (4.3)
, $||T_{m(t)}^{m(t+\epsilon)}(x)-x||_{L_{m(t)}^{2}}=W_{2}(m(t),m(t+\epsilon)),$
we see that. for any absolutely continuous curve 

\begin{equation}
W_{2}(m(t),m(t+\epsilon))\leq C(t)\epsilon\label{eq:4-60}
\end{equation}
In the definition of the functional derivative, see (\ref{eq:4-22})
we can write 

\begin{equation}
\frac{F(m_{\epsilon})-F(m)-\int_{R^{n}}\dfrac{\partial F(m)}{\partial m}(x)(m_{\epsilon}(x)-m(x))dx}{\epsilon}\rightarrow0,\:\text{as}\:\epsilon\rightarrow0\label{eq:4-22-1}
\end{equation}
provided the sequence $m_{\epsilon}$ is absolutely continuous. 

Suppose that we consider the curve corrresponding to a gradient $D\Phi(x)$
where $\Phi(x)$ is smooth with compact support, i.e the curve $m(t)$
is defined by 

\begin{align}
\frac{\partial m}{\partial t}+\text{div }(D\Phi(x)m(x,t)) & =0\label{eq:4-7}\\
m(x,0)=m(x)\nonumber 
\end{align}
Since it is a gradient, $D\Phi(x)$ has minimal norm and we can claim
from (\ref{eq:4-6}) that 

\begin{equation}
D\Phi(x)=\lim_{\epsilon\rightarrow0}\frac{T_{m}^{m(\epsilon)}(x)-x}{\epsilon}\:\text{in}\:L_{m}^{2}(R^{n};R^{n})\label{eq:4-8}
\end{equation}
We consider now a functional $F(m)$ on $\mathcal{P}_{2}(R^{n})$,
and limit ourselves to densities. We say that $F(m)$ is differentiable
at $m$ if there exists a function $\Gamma(x,m)$ belonging to the
tangent space $\mathcal{T}(m)$ with the property 

\begin{equation}
\frac{F(m(\epsilon))-F(m)-\int_{R^{n}}\Gamma(x,m).(T_{m}^{m(\epsilon)}(x)-x)m(x)dx}{W_{2}(m,m(\epsilon))}\rightarrow0,\,\text{as }\epsilon\rightarrow0\label{eq:4-9}
\end{equation}

We recall that , see ( \ref{eq:4-3}) $W_{2}(m,m(\epsilon))=||T_{m}^{m(\epsilon)}(x)-x||_{L_{m}^{2}}.$
The function $\Gamma(x,m)$ is called the Wasserstein gradient and
denoted $\nabla F_{m}(m)(x).$ If we apply this property to the map
$m(t)$ defined by (\ref{eq:4-7}), this is equivalent to 

\[
\dfrac{F(m(\epsilon))-F(m)}{\epsilon}\rightarrow\int_{R^{n}}\Gamma(x,m).D\Phi(x)m(x)dx
\]

From the continuity equation (\ref{eq:4-7}), using the regularity
of $\Phi,$ we can state that

\[
\dfrac{m(x,\epsilon)-m(x)}{\epsilon}\rightarrow-\text{div }(D\Phi(x)m(x)),\:\text{as }\epsilon\rightarrow0,\:\text{in the sense of distributions}
\]
If $F(m)$ has a functional derivative we obtain 

\[
\dfrac{F(m(\epsilon))-F(m)}{\epsilon}\rightarrow-\int_{R^{n}}\frac{\partial F(m)}{\partial m}(x)\text{div }(D\Phi(x)m(x))dx
\]
Therefore we obtain 

\begin{align*}
\int_{R^{n}}\Gamma(x,m).D\Phi(x)m(x)dx & =-\int_{R^{n}}\frac{\partial F(m)}{\partial m}(x)\text{div }(D\Phi(x)m(x))dx\\
 & =\int_{R^{n}}D\frac{\partial F(m)}{\partial m}(x).D\Phi(x)m(x)dx
\end{align*}
 If we assume that $D\frac{\partial F(m)}{\partial m}(x)\in L_{m}^{2}(R^{n};R^{n}),$
we can replace $D\Phi(x)$ by any element of $\mathcal{T}(m)$. Since
$\Gamma(x,m)$ and $D\frac{\partial F(m)}{\partial m}(x)$ belong
to $\mathcal{T}(m),$ it follows that 

\begin{equation}
\nabla F_{m}(m)(x)=D\frac{\partial F(m)}{\partial m}(x)\label{eq:4-10}
\end{equation}
So the Wasserstein gradient is simply the gradient of the functional
derivative. 
\begin{rem}
\label{rem3-1} The concept of functional derivative, defined in (\ref{eq:4-22})
uses a sequence of probability densities $m_{\epsilon}\rightarrow m,$
so it is not equivalent to the concept of Gâteaux differential in
the space $L^{2}(R^{n}),$ which requires to remove the assumptions
of positivity and $\int_{R^{n}}m(x)dx=1.$ We will develop the differences
in examples in which explicit formulas are available, see section
\ref{sec:QUADRATIC-CASE}.
\end{rem}

\subsection{\label{sub:GRADIENT-IN-THE}GRADIENT IN THE HILBERT SPACE $\mathcal{H}$.}

The functional $F(m)$ can now be written as a functional $\mathcal{F}(X)$
on $\mathcal{H}$ , with $m=\mathcal{L}_{X}.$ We assume that random
variables with densities form a dense subspace of $\mathcal{H}.$
Consider a random variable $Y\in\mathcal{H}$ and let $\pi(x,y)$
be the joint probability density on $R^{n}\times R^{n}$ of the pair
$(X,Y).$ So $m(x)=\int_{R^{n}}\pi(x,y)dy.$ Consider then the random
variable $X+tY.$ Its probability density is given by 

\[
m(x,t)=\int_{R^{n}}\pi(x-ty,y)dy
\]
 and it sastisfies the continuity equation 

\[
\frac{\partial m}{\partial t}=-\text{div }(\int_{R^{n}}\pi(x-ty,y)ydy)
\]
We have $\mathcal{F}(X+tY)=F(m(t)).$ Next 

\[
\lim_{t\rightarrow0}\frac{\mathcal{F}(X+tY)-\mathcal{F}(X)}{t}=((D_{X}\mathcal{F}(X),Y))
\]
 and 

\begin{align*}
\lim_{t\rightarrow}\frac{F(m(t))-F(m)}{t} & =-\int_{R^{n}}\frac{\partial F(m)}{\partial m}(x)\text{div }(\int_{R^{n}}\pi(x,y)ydy)dx\\
 & =\int_{R^{n}}D\frac{\partial F(m)}{\partial m}(x).(\int_{R^{n}}\pi(x,y)ydy)dx\\
 & =((D\frac{\partial F(m)}{\partial m}(X),Y))
\end{align*}
 Thus necessarily 

\begin{equation}
D\frac{\partial F(m)}{\partial m}(X)=\nabla F_{m}(m)(X)=D_{X}\mathcal{F}(X)\label{eq:4-11}
\end{equation}
So, the gradient in $\mathcal{H}$ reduces to the Wasserstein gradient,
in which the argument is replaced with the random variable. In the
sequel, we will use the gradient in $\mathcal{H}$.

\section{MEAN FIELD TYPE CONTROL PROBLEM}

\subsection{PRELIMINARIES}

Consider a function $f(x,m)$ defined on $R^{n}\times\mathcal{P}_{2}(R^{n})$.
As usual we consider only $m$ which are densities of probability
measures, and use also the notation $f(x,\mathcal{L}_{X}).$ We then
define $\mathcal{F}(X)=Ef(X,\mathcal{L}_{X}).$ This implies 

\begin{equation}
\mathcal{F}(X)=\Phi(m)=\int_{R^{n}}f(x,m)m(x)dx\label{eq:5-1}
\end{equation}
We next consider the functional derivative 

\begin{equation}
\frac{\partial\Phi(m)}{\partial m}(x)=F(x,m)=f(x,m)+\int_{R^{n}}\frac{\partial f}{\partial m}(\xi,m)(x)m(\xi)d\xi\label{eq:5-2}
\end{equation}
and we have 

\begin{equation}
D_{X}\mathcal{F}(X)=D_{x}F(X,\mathcal{L}_{X})\label{eq:5-3}
\end{equation}
We make the assumptions 

\begin{equation}
|D_{x}F(x,m)|\leq\dfrac{c}{2}(1+|x|+(\int|\xi|^{2}m(\xi)d\xi)^{\frac{1}{2}})\label{eq:5-4}
\end{equation}

\begin{equation}
|D_{x}F(x_{1},m_{1})-D_{x}F(x_{2},m_{2})|\leq\dfrac{c}{2}(|x_{1}-x_{2}|+W_{2}(m_{1},m_{2}))\label{eq:5-5}
\end{equation}
which implies immediately the properties (\ref{eq:2-1}), (\ref{eq:2-1-1}).

\subsection{EXAMPLES}

We consider first quadratic functionals. We use the notation $\bar{x}=\int_{R^{n}}xm(x)dx.$
We then consider 

\begin{equation}
f(x,m)=\frac{1}{2}(x-S\bar{x})^{*}\bar{Q}(x-S\bar{x})+\frac{1}{2}x^{*}Qx\label{eq:5-53}
\end{equation}
then assuming that $\int_{R^{n}}m(x)dx=1,$ i.e. $m$ is a probability
density we have 

\begin{equation}
F(x,m)=\frac{1}{2}x^{*}(Q+\bar{Q})x+\frac{1}{2}\bar{x}^{*}S^{*}\bar{Q}S\bar{x}-\bar{x}^{*}(\bar{Q}S+S^{*}\bar{Q}-S^{*}\bar{Q}S)x\label{eq:5-54}
\end{equation}

\begin{equation}
D_{x}F(x,m)=(Q+\bar{Q})x-(\bar{Q}S+S^{*}\bar{Q}-S^{*}\bar{Q}S)\bar{x}\label{eq:5-55}
\end{equation}

We see easily that assumptions (\ref{eq:5-4}),(\ref{eq:5-5}) are
satisfied.

We can give an additonal example

\begin{equation}
f(x,m)=\frac{1}{2}\int_{R^{n}}K(x,\xi)m(\xi)d\xi\label{eq:5-6}
\end{equation}

with $K(x,\xi)=K(\xi,x)$ and 

\begin{equation}
|K(x_{1},\xi_{1})-K(x_{2},\xi_{2})|\leq C(1+|x_{1}|+|x_{2}|+|\xi_{1}|+|\xi_{2}|)(|x_{1}-x_{2}|+|\xi_{1}-\xi_{2}|)\label{eq:5-67}
\end{equation}

\begin{align}
|D_{x}K(x_{1},\xi_{1})-D_{x}K(x_{2},\xi_{2})| & \leq\dfrac{c}{2}(|x_{1}-x_{2}|+|\xi_{1}-\xi_{2}|)\label{eq:5-7}\\
|D_{x}K(0,0)| & \leq\dfrac{c}{2}\nonumber 
\end{align}

We have 

\[
\frac{\partial f}{\partial m}(\xi,m)(x)=\frac{1}{2}K(\xi,x)=\frac{1}{2}K(x,\xi)
\]

hence $\int_{R^{n}}\frac{\partial f}{\partial m}(\xi,m)(x)m(\xi)d\xi=f(x,m)$
which implies 

\begin{equation}
F(x,m)=2f(x,m)=\int_{R^{n}}K(x,\xi)m(\xi)d\xi\label{eq:5-7-1}
\end{equation}

We thus have 

\begin{align*}
|D_{x}F(x,m)|| & \leq\int_{R^{n}}|D_{x}K(x,\xi)|m(\xi)d\xi\leq\\
 & \leq\dfrac{c}{2}(1+|x|+\int|\xi|m(\xi)d\xi)\leq\\
 & \leq\dfrac{c}{2}(1+|x|+(\int|\xi|^{2}m(\xi)d\xi)^{\frac{1}{2}})
\end{align*}
 If we take 2 densities $m_{1},m_{2},$ we may consider 2 random variables
$\Xi_{1},\Xi_{2}$ with the probabilities $m_{1},m_{2}.$ Therefore 

\begin{align*}
|D_{x}F(x_{1},m_{1})-D_{x}F(x_{2},m_{2})| & \leq|\int D_{x}(K(x_{1},\xi)-K(x_{2},\xi))m_{1}(\xi)d\xi||+\\
+ & |E\:D_{x}(K(x_{2},\Xi_{1})-K(x_{2},\Xi_{2}))|\leq\\
\dfrac{c}{2}|x_{1}-x_{2}| & +\dfrac{c}{2}\sqrt{E|\Xi_{1}-\Xi_{2}|^{2}}
\end{align*}
 and since $\Xi_{1},\Xi_{2}$ are arbitrary, with marginals $m_{1},m_{2}$we
can write (\ref{eq:5-5}). In the sequel we also consider a functional
$h(x,m)$ with exactly the same properties as $f$ and write 

\begin{align}
F_{T}(x,m) & =h(x,m)+\int_{R^{n}}\frac{\partial h}{\partial m}(\xi,m)(x)m(\xi)d\xi\label{eq:5-8}\\
\mathcal{F}_{T}(X) & =\int_{R^{n}}h(x,m)m(x)dx,\;D_{X}\mathcal{F}_{T}(X)=D_{x}F_{T}(X,\mathcal{L}_{X})\nonumber 
\end{align}

\subsection{\label{sub:MEAN-FIELD-TYPE}MEAN FIELD TYPE CONTROL PROBLEM}

We can formulate the following mean field type control problem. Let
us consider a dynamical system in $R^{n}$

\begin{align}
\frac{dx}{ds} & =v(x(s),s)\label{eq:5-9}\\
x(t) & =\xi\nonumber 
\end{align}
where $v(x,s)$ is a feedback to be optimized. The initial condition
is a random variable with probability density $m(x).$The Fokker-Planck
equation of the evolution of the density is 

\begin{align}
\frac{\partial m}{\partial s}+\text{div}(v(x)m) & =0\label{eq:5-10}\\
m(x,t)=m(x)\nonumber 
\end{align}
We denote the solution by $m_{v(.)}(x,s)$. Similarly we call the
solution of (\ref{eq:5-9}) $x(s;v(.)).$ We then consider the cost
functional 

\begin{align}
J_{m,t}(v(.)) & =\frac{\lambda}{2}\int_{t}^{T}\int_{R^{n}}|v(x,s)|^{2}m_{v(.)}(x,s)dxds+\int_{t}^{T}\int_{R^{n}}m_{v(.)}(x,s)f(x,m_{v(.)}(s))dxds+\label{eq:5-11}\\
 & +\int_{R^{n}}m_{v(.)}(x,T)h(x,m_{v(.)}(T))dx\nonumber 
\end{align}
which is equivalent to the expression 

\begin{equation}
J_{m,t}(v(.))=\frac{\lambda}{2}\int_{t}^{T}E|v(x(s;v(.)))|^{2}ds+\int_{t}^{T}Ef(x(s;v(.)),\mathcal{L}_{x(s)})ds+Eh(x(T;v(.)),\mathcal{L}_{x(T)})\label{eq:5-12}
\end{equation}
This is a standard mean field type control problem, not a mean field
game. In \cite{BFY} we have associated to it a coupled system of
HJB and FP equations, see p. 18, which reads here 

\begin{align}
-\frac{\partial u}{\partial s}+\frac{1}{2\lambda}|Du|^{2} & =F(x,m(s))\label{eq:5-13}\\
u(x,T) & =F_{T}(x,m(s))\nonumber \\
\frac{\partial m}{\partial s}-\frac{1}{\lambda}\text{div} & (Du\,m)=0\nonumber \\
m(x,t) & =m(x)\nonumber 
\end{align}
 This system expresses a necessary condition of optimality. The function
$u(x,t)$ is not a value function, but an adjoint variable to the
optimal state, which is $m(x,s)$. The optimal feedback is given by 

\begin{equation}
\hat{v}(x,s)=-\frac{1}{\lambda}Du(x,s)\label{eq:5-14}
\end{equation}
We proceed formally, although we shall be able to give an explicit
solution of this system. If $\hat{v}(x,s)$ is the optimal feedback,
then the value function $V(m,t)=J_{m,t}(\hat{v}(.))$ is given by 

\begin{align}
V(m,t) & =\frac{1}{2\lambda}\int_{t}^{T}\int_{R^{n}}m(x,s)|Du(x,s)|^{2}dxds+\int_{t}^{T}\int_{R^{n}}m(x,s)f(x,m(s))dxds+\label{eq:5-15}\\
+ & \int_{R^{n}}m(x,T)h(x,m(T))dx\nonumber 
\end{align}
The value function is solution of Bellman equation, see \cite{BFY1},
\cite{LPI}, written formally (it will be justified later)

\begin{align}
\frac{\partial V}{\partial t}-\frac{1}{2\lambda}\int_{R^{n}}|D_{\xi}\frac{\partial V(m,t)}{\partial m}(\xi)|^{2}m(\xi)d\xi+\int_{R^{n}}f(\xi,m)m(\xi)d\xi & =0\label{eq:5-16}\\
V(m,T)=\int_{R^{n}}h(\xi,m)m(\xi)d\xi\nonumber 
\end{align}

\subsection{SCALAR MASTER EQUATION}

We derive the master equation, by considering the function 

\[
U(x,m,t)=\frac{\partial V(m,t)}{\partial m}(x)
\]
and we note that 

\[
\frac{\partial U}{\partial m}(x,m,t)(\xi)=\frac{\partial^{2}V(m,t)}{\partial m^{2}}(x,\xi)
\]
therefore the function is symmetric in $x,\xi$ which means

\[
\frac{\partial U}{\partial m}(x,m,t)(\xi)=\frac{\partial U}{\partial m}(\xi,m,t)(x)
\]
By differentiating (\ref{eq:5-16}) in $m$, and using the symmetry
property, we obtain the equation 

\begin{align}
\frac{\partial U}{\partial t}-\frac{1}{\lambda}\int_{R^{n}}D_{\xi}\frac{\partial U}{\partial m}(x,m,t)(\xi).D_{\xi}U(\xi,m,t)m(\xi)d\xi-\label{eq:5-17}\\
-\frac{1}{2\lambda}|D_{x}U(x,m,t)|^{2}+F(x,m)=0\nonumber \\
U(x,m,T)=F_{T}(x,m)\nonumber 
\end{align}
This function allows to uncouple the system of HJB-FP equations, given
in (\ref{eq:5-13}). Indeed, we first solve the FP equation, replacing
$u$ by $U,$ i.e. 

\begin{align}
\frac{\partial m}{\partial s}-\frac{1}{\lambda}\text{div}(DU\,m) & =0\label{eq:5-18}\\
m(x,t) & =m(x)\nonumber 
\end{align}
 then $u(x,s)=U(x,m(s),s)$ is solution of the HJB equation (\ref{eq:5-13}),
as easily checked. In particular , we have 

\begin{equation}
u(x,t)=U(x,m,t)\label{eq:5-19}
\end{equation}

\subsection{VECTOR MASTER EQUATION}

We next consider $\mathcal{U}(x,m,t)=D_{x}U(x,m,t).$ Differentiating
(\ref{eq:5-17}) we can write 

\begin{equation}
\frac{\partial\mathcal{U}}{\partial t}-\frac{1}{\lambda}\int_{R^{n}}D_{\xi}\frac{\partial\mathcal{U}}{\partial m}(x,m,t)(\xi)\,\mathcal{U}(\xi,m,t)m(\xi)d\xi-\frac{1}{\lambda}D_{x}\mathcal{U}(x,m,t)\,\mathcal{U}(x,m,t)+D_{x}F(x.m)=0\label{eq:5-20}
\end{equation}
\[
\mathcal{U}(x,m,T)=D_{x}F_{T}(x.m)
\]

\section{CONTROL PROBLEM IN THE SPACE $\mathcal{H}$.}

\subsection{FORMULATION}

If we set 

\begin{align}
\mathcal{F}(X) & =Ef(X,\mathcal{L}_{X})=\int f(x,m)m(x)dx\label{eq:6-1}\\
\mathcal{F}_{T}(X) & =Eh(X,\mathcal{L}_{X})=\int h(x,m)m(x)dx\nonumber 
\end{align}

\begin{align}
F(x,m) & =f(x,m)+\int\frac{\partial f(\xi,m)}{\partial m}(x)m(\xi)d\xi\label{eq:6-2}\\
F_{T}(x,m) & =h(x,m)+\int\frac{\partial h(\xi,m)}{\partial m}(x)m(\xi)d\xi\nonumber 
\end{align}
We have 

\begin{align}
D_{X}\mathcal{F}(X) & =D_{x}F(X,\mathcal{L}_{X})\label{eq:6-3}\\
D_{X}\mathcal{F}_{T}(X) & =D_{x}F_{T}(X,\mathcal{L}_{X})\nonumber 
\end{align}
We assume that 

\begin{align}
|D_{x}F(x_{1},m_{1})-D_{x}F(x_{2},m_{2})| & \leq\frac{c}{2}(|x_{1}-x_{2}|+W_{2}(m_{1},m_{2}))\label{eq:6-4}\\
|D_{x}F_{T}(x_{1},m_{1})-D_{x}F_{T}(x_{2},m_{2})| & \leq\frac{c}{2}(|x_{1}-x_{2}|+W_{2}(m_{1},m_{2}))\nonumber 
\end{align}

\begin{align}
|D_{x}F(x,m)| & \leq\dfrac{c}{2}(1+|x|+\sqrt{\int_{R^{n}}|\xi|^{2}m(\xi)d\xi})\label{eq:6-4-1}\\
|D_{x}F_{T}(x,m)| & \leq\dfrac{c}{2}(1+|x|+\sqrt{\int_{R^{n}}|\xi|^{2}m(\xi)d\xi})\nonumber 
\end{align}

It follows that 

\begin{align*}
||D_{X}\mathcal{F}(X_{1})-D_{X}\mathcal{F}(X_{2})|| & \leq||D_{x}F(X_{1},\mathcal{L}_{X_{1}})-D_{x}F(X_{2},\mathcal{L}_{X_{1}})||+\\
+ & ||D_{x}F(X_{2},\mathcal{L}_{X_{1}})-D_{x}F(X_{2},\mathcal{L}_{X_{2}})||
\end{align*}

\[
\leq c||X_{1}-X_{2}||
\]
 and similar estimate for $\mathcal{F}_{T}.$Therefore the set up
of section \ref{sub:SETTING-OF-THE} is satisfied. We can reinterpret
the problem (\ref{eq:5-9}), (\ref{eq:5-12}) or (\ref{eq:5-10}),
(\ref{eq:5-11}) as (\ref{eq:2-2}), (\ref{eq:2-3}) which has been
completely solved in Theorem \ref{theo2-1}. We shall study the solution
of the abstract setting. Of course, the initial state $X$ has probability
law $\mathcal{L}_{X}=m.$

\subsection{\label{sub:INTERPRETATION-OF-THE}INTERPRETATION OF THE SOLUTION }

The key point of the proof of Theorem \ref{theo2-1} is the study
of the system (\ref{eq:2-11}) which has one and only one solution.
We proceed formally. Consider the HJB-FP system (\ref{eq:5-13}).
The initial conditions are the pair $(m,t),$ so we can write the
solution as $u_{m,t}(x,s)$, $m_{m,t}(x,s).$ We introduce the differential
equation 

\begin{align}
\frac{dy}{ds} & =-\frac{1}{\lambda}Du(y(s),s)\label{eq:6-5}\\
y(t) & =x\nonumber 
\end{align}
The solution ( if it exists) can be written $y_{xmt}(s).$Now let
us set $z_{xmt}(s)=Du_{mt}(y_{xmt}(s),s).$ Differentiating the HJB
equation (\ref{eq:5-13}) and computing the derivative $\dfrac{dz}{ds}$
we obtain 

\begin{align}
-\dfrac{dz}{ds} & =D_{x}F(y(s),m(s))\label{eq:6-6}\\
z(T) & =D_{x}F_{T}(y(T),m(T))\nonumber 
\end{align}
Now, from the definition of $m(s)$ solution of the FP equation, we
can write 

\begin{equation}
m(s)=y_{mt}(s)(.)(m)\label{eq:6-7}
\end{equation}
in which we have used the notation $y_{mt}(s)(x)=y_{mt}(x,s)=y_{xmt}(s)$
and $y_{mt}(s)(.)(m)$ means the image measure of $m$ by the map
$y_{mt}(s)(.).$ So we can write the system (\ref{eq:6-5}), (\ref{eq:6-6})
as 

\begin{align}
\frac{d^{2}y}{ds^{2}} & =\frac{1}{\lambda}D_{x}F(y(s),y(s)(.)(m))\label{eq:6-8}\\
y(t)=x & \;\frac{dy}{ds}(T)=-\frac{1}{\lambda}D_{x}F_{T}(y(T),y(T)(.)(m))\nonumber 
\end{align}
 This is also written in integral form 

\begin{align}
y(s) & =x-\frac{s-t}{\lambda}D_{x}F_{T}(y(T),y(T)(.)(m))\label{eq:6-9}\\
- & \frac{1}{\lambda}\int_{t}^{T}D_{x}F(y(\sigma),y(\sigma)(.)(m))(s\wedge\sigma-t)d\sigma\nonumber 
\end{align}
Now if we take $y_{X,\mathcal{L}_{X},t}(s),$ then $y(s)(.)(\mathcal{L}_{X})=\mathcal{L}_{y(s)}$
. Writing $y_{X,\mathcal{L}_{X},t}(s)=Y(s)$ to emphasize that we
are dealing with a random variable, we can write (\ref{eq:6-9}) as 

\begin{align}
Y(s) & =X-\frac{s-t}{\lambda}D_{x}F_{T}(Y(T),\mathcal{L}_{Y(T)})-\label{eq:6-10}\\
- & \frac{1}{\lambda}\int_{t}^{T}D_{x}F(Y(\sigma),\mathcal{L}_{Y(\sigma)})(s\wedge\sigma-t)d\sigma\nonumber 
\end{align}
 which is nothing else than (\ref{eq:2-3}) recalling the values of
$D_{X}\mathcal{F}(X),\:D_{X}\mathcal{F}_{T}(X)$, cf (\ref{eq:6-3}).
We know from Theorem \ref{theo2-1}that (\ref{eq:6-10}) has one and
only one solution in $C^{0}([t,T];\mathcal{H})$ and in fact in $C^{2}([t,T];\mathcal{H}).$
This result, of course, does not allow to go from (\ref{eq:6-10})
to (\ref{eq:6-9}), but it easy to mimic the proof. We state the result
in the following 
\begin{prop}
\label{prop3}We assume (\ref{eq:6-4}),(\ref{eq:6-4-1}) and condition
(\ref{eq:2-5}). For given $m,t$ there exists one and only one solution
$y_{mt}(x,s)$ of (\ref{eq:6-9}) in the space $C(t,T;L_{m}^{2}(R^{n};R^{n})).$ \end{prop}
\begin{proof}
We use a fixed point argument. We define a map from $C(t,T;L_{m}^{2}(R^{n};R^{n}))$
to itself. Let $z(x,s)$ a function in $C(t,T;L_{m}^{2}(R^{n};R^{n})).$
We define 

\begin{align*}
\zeta(x,s) & =x-\frac{s-t}{\lambda}D_{x}F_{T}(z(x,T),z(T)(.)(m))-\\
- & \frac{1}{\lambda}\int_{t}^{T}D_{x}F(z(x,\sigma),z(\sigma)(.)(m))(s\wedge\sigma-t)d\sigma
\end{align*}
 We have 

\begin{align*}
|\zeta(x,s)| & \leq|x|+\dfrac{T}{\lambda}\dfrac{c}{2}(1+|z(x,T)|+(\int_{R^{n}}|z(\xi,T)|^{2}m(\xi)d\xi)^{1/2})+\\
+ & \dfrac{cT}{2\lambda}\int_{t}^{T}(1+|z(x,\sigma)|+(\int_{R^{n}}|z(\xi,\sigma)|^{2}m(\xi)d\xi)^{1/2})d\sigma
\end{align*}
hence , from norm properties 

\begin{align*}
\sqrt{\int_{R^{n}}|\zeta(x,s)|^{2}m(x)dx} & \leq\sqrt{\int_{R^{n}}|x|^{2}m(x)dx}+\dfrac{T}{\lambda}\dfrac{c}{2}(1+2(\int_{R^{n}}|z(\xi,T)|^{2}m(\xi)d\xi)^{1/2})+\\
+ & \dfrac{cT}{2\lambda}\int_{t}^{T}(1+2(\int_{R^{n}}|z(\xi,\sigma)|^{2}m(\xi)d\xi)^{1/2})d\sigma
\end{align*}
 and we conclude easily that $\zeta$ belongs to $C(t,T;L_{m}^{2}(R^{n};R^{n})).$
We set $\zeta=\mathcal{T}$(z). Using the assumptions and similar
estimates, one checks that $\mathcal{T}$ is a contraction. We prove
indeed that 

\begin{equation}
||\mathcal{T}(z_{1})-\mathcal{T}(z_{2})||_{C(t,T;L_{m}^{2})}\leq(1-\dfrac{cT(T+1)}{\lambda})||z_{1}-z_{2}||_{C(t,T;L_{m}^{2})}\label{eq:6-11}
\end{equation}

$\blacksquare$ 
\end{proof}
It follows immediately that the solution $y_{xmt}(s)=y_{mt}(x,s)$
satisfies the estimate 

\begin{equation}
||y_{mt}(.)||_{C(t,T;L_{m}^{2})}\leq\frac{\lambda\sqrt{\int_{R^{n}}|x|^{2}m(x)dx}+cT(T+1)}{\lambda-cT(T+1)}\label{eq:6-12}
\end{equation}
Since $Y_{Xt}(s)=y_{X,\mathcal{L}_{X},t}(s)$ we deduce the first
estimate (\ref{eq:2-15}). We consider next 

\begin{align}
z_{xmt}(s) & =z_{mt}(x,s)=D_{x}F_{T}(y(x,T),y(T)(.)(m))+\label{eq:6-12-1}\\
+ & \int_{s}^{T}D_{x}F(y(x,\sigma),y(\sigma)(.)(m))d\sigma\nonumber 
\end{align}
and from the assumption (\ref{eq:6-4-1}) we obtain easily 

\[
||z_{mt}(.)||_{C(t,T;L_{m}^{2})}\leq c(1+T)(1+||y_{mt}(.)||_{C(t,T;L_{m}^{2})})
\]
hence

\begin{equation}
||z_{mt}(.)||_{C(t,T;L_{m}^{2})}\leq\lambda c(1+T)\frac{\sqrt{\int_{R^{n}}|x|^{2}m(x)dx}+1}{\lambda-cT(T+1)}\label{eq:6-13}
\end{equation}

Clearly $Z(s)=Z_{Xt}(s)=z_{X,\mathcal{L}_{X},t}(s)$, see (\ref{eq:2-11}),
and we recover the 2nd estimate (\ref{eq:2-15}). 

We can give more properies on $y_{xmt}(s).$ We write first 

\begin{align*}
y_{mt}(x_{1},s)-y_{mt}(x_{2},s) & =x_{1}-x_{2}-\frac{s-t}{\lambda}(D_{x}F_{T}(y_{mt}(x_{1},T),y_{mt}(T)(.)(m))-D_{x}F_{T}(y_{mt}(x_{2},T),y_{mt}(T)(.)(m)))-\\
- & \frac{1}{\lambda}\int_{t}^{T}(D_{x}F(y_{mt}(x_{1},\sigma),y_{mt}(\sigma)(.)(m))-D_{x}F(y_{mt}(x_{2},\sigma),y_{mt}(\sigma)(.)(m)))(s\wedge\sigma-t)d\sigma
\end{align*}
 From (\ref{eq:6-4}) we obtain easily 

\begin{equation}
\sup_{t<s<T}|y_{mt}(x_{1},s)-y_{mt}(x_{2},s)\leq\frac{\lambda|x_{1}-x_{2}|}{\lambda-cT(T+1)}\label{eq:6-14}
\end{equation}

Also 

\begin{equation}
\sup_{t<s<T}|y_{mt}(x,s)|\leq\lambda\frac{|x|+\dfrac{Tc(1+T)(1+\sqrt{\int_{R^{n}}|\xi|^{2}m(\xi)d\xi})}{\lambda-cT(T+1)}}{\lambda-cT(T+1)}\label{eq:6-15}
\end{equation}

A similar estimate holds for $\sup_{t<s<T}|z_{mt}(x,s)|$.

\section{BELLMAN EQUATION AND MASTER EQUATION}

\subsection{THE VALUE FUNCTION }

The value function of the control problem in $\mathcal{H}$ is given
by 

\[
V(X,t)=\frac{1}{2\lambda}\int_{t}^{T}||Z(s)||^{2}ds+\int_{t}^{T}\mathcal{F}(Y(s))ds+\mathcal{F}_{T}(Y(T))
\]
in which $Y(s)=$$y_{X,\mathcal{L}_{X},t}(s)$ and $Z(s)=z_{X,\mathcal{L}_{X},t}(s).$
From this representation and the definition of $\mathcal{F}$ and
$\mathcal{F}_{T}$ we can assert that $V(X,t)$ depends only on $\mathcal{L}_{X}$
and thus can be written $V(m,t)$ with 

\begin{align}
V(m,t) & =\frac{1}{2\lambda}\int_{t}^{T}\int_{R^{n}}|z_{xmt}(s)|^{2}m(x)dxds+\int_{t}^{T}\int_{R^{n}}f(y_{xmt}(s),y_{mt}(s)(.)(m))m(x)dxds+\label{eq:7-1}\\
 & +\int_{R^{n}}h(y_{xmt}(T),y_{mt}(T)(.)(m))m(x)dx\nonumber 
\end{align}
From (\ref{eq:6-13}) we have 

\[
\int_{t}^{T}\int_{R^{n}}|z_{xmt}(s)|^{2}m(x)dxds\leq\dfrac{T\lambda^{2}c^{2}(1+T)^{2}(1+\int_{R^{n}}|x|^{2}m(x)dx)}{(\lambda-cT(T+1))^{2}}
\]
 and $|\mathcal{F}(Y(s))|\leq C(1+||Y(s)||^{2}),$ therefore 

\[
|\int_{R^{n}}f(y_{xmt}(s),y_{mt}(s)(.)(m))m(x)dx|\leq C(1+\int|y_{xmt}(s)|^{2}m(x)dx)
\]
 and from the estimate (\ref{eq:6-12}) we obtain 

\[
|\int_{t}^{T}\int_{R^{n}}f(y_{xmt}(s),y_{mt}(s)(.)(m))m(x)dxds|\leq CT[1+\dfrac{\lambda^{2}\int_{R^{n}}|x|^{2}m(x)dx+T^{2}(T+1)^{2}}{(\lambda-cT(T+1))^{2}}]
\]
 and the third term in the right hand side of $(\ref{eq:7-1}$) satisfies
a similar estimate. We thus have obtained 

\begin{equation}
|V(m,t)|\leq C(1+\int_{R^{n}}|x|^{2}m(x)dx)\label{eq:7-2}
\end{equation}
which is, of course, equivalent to the 1st estimate (\ref{eq:2-6}).

We turn now to $U(x,m,t)=$$\dfrac{\partial V(m,t)}{\partial m}(x)$.
We have seen formally in (\ref{eq:5-19}) that $U(x,m,t)=u(x,t)=u_{mt}(x,t).$
We need to prove it. We begin by giving a solution to the system HJB-FP
equations (\ref{eq:5-13}). We have the 
\begin{lem}
\label{lem4} We make the assumptions of Proposition \ref{prop3}.
We can give an explicit formula to the system (\ref{eq:5-13}). We
have 
\begin{align}
u_{mt}(x,t) & =\dfrac{1}{2\lambda}\int_{t}^{T}|z_{xmt}(s)|^{2}ds+\int_{t}^{T}F(y_{xmt}(s),y_{mt}(s)(.)(m))ds+\label{eq:7-3}\\
 & +F_{T}(y_{xmt}(T),y_{mt}(T)(.)(m))\nonumber 
\end{align}

and $m_{mt}(s)=y_{mt}(s)(.)(m).$\end{lem}
\begin{proof}
Indeed, if we look at $F(x,m(s))$ and $F_{T}(x,m(T))$ in which $m(.)$
is frozen, the HJB equation appears as a standard one for a deterministic
control problem. This problem is simply 
\begin{align*}
\dfrac{dx}{ds} & =v(s)\\
x(t) & =x
\end{align*}
\[
J_{xt}(v(.))=\frac{\lambda}{2}\int_{t}^{T}|v(s)|^{2}ds+\int_{t}^{T}F(x(s),m(s))ds+F_{T}(x(T),m(T))
\]
 in which the function $m(s)$ is frozen, but not arbitrary. It is
the function solution of the FP equation, in the system (\ref{eq:5-13})
If we write the necessary conditions of optimality, one checks easily
that in view of the specific value of $m(s),$the optimal state is
$y_{xmt}(s)$ and the optimal control is $-\dfrac{1}{\lambda}z_{xmt}(s).$
In plugging these values in the cost function, we obtain formula (\ref{eq:7-3}).
$\blacksquare$ 
\end{proof}
We may assume that 

\begin{align}
|F(x,m)| & ,|F_{T}(x,m)|\leq C(1+|x|^{2}+\int|\xi|^{2}m(\xi)d\xi)\label{eq:7-4}
\end{align}
We shall also assume that 

\begin{equation}
|\frac{\partial F(x,m)}{\partial m}(\xi)|,\:|\frac{\partial F_{T}(x,m)}{\partial m}(\xi)|\leq C(1+|x|^{2}+|\xi|^{2}+\int|\eta|^{2}m(\eta)d\eta)\label{eq:7-41}
\end{equation}

\begin{equation}
|D_{x}D_{\xi}\frac{\partial F(x,m)}{\partial m}(\xi)|\leq C\label{eq:7-42}
\end{equation}
We also make an assumption which simplifies proofs, but which can
be overcome, with technical difficulties. 

\begin{align}
\int_{R^{n}}(F(x,m_{1})-F(x,m_{2})(m_{1}(x)-m_{2}(x))dx & \geq0\label{eq:7-43}\\
\int_{R^{n}}(F_{T}(x,m_{1})-F_{T}(x,m_{2})(m_{1}(x)-m_{2}(x))dx & \geq0\nonumber 
\end{align}
This assumption allows to obtain the following interesting in itself
result
\begin{prop}
\label{prop5-1} We assume (\ref{eq:7-43}). Then considering the
system of HJB-FP equations (\ref{eq:5-13}) with initial conditions
$m_{1}(x)$ and $m_{2}(x)$ and calling $u_{1}(x,s),m_{1}(x,s)$,
respectively $u_{2}(x,s),m_{2}(x,s)$ the solutions, we have the property 

\begin{equation}
\int_{R^{n}}(u_{1}(x,t)-u_{2}(x,t))(m_{1}(x)-m_{2}(x))dx\geq0\label{eq:7-44}
\end{equation}
\end{prop}
\begin{proof}
From the system HJB-FP we can write 

\begin{align*}
-\frac{\partial}{\partial s}(u_{1}-u_{2})+\frac{1}{2\lambda}|Du_{1}|^{2}-\frac{1}{2\lambda}|Du_{2}|^{2} & =F(x,m_{1}(s))-F(x,m_{2}(s))\\
u_{1}(x,T)-u_{2}(x,T) & =F_{T}(x,m_{1}(T))-F_{T}(x,m_{2}(T))
\end{align*}
\begin{align*}
\frac{\partial}{\partial s}(m_{1}-m_{2}) & =\frac{1}{\lambda}\text{div}(Du_{1}m_{1}-Du_{2}m_{2})\\
m_{1}(x,t) & -m_{2}(x,t)=m_{1}(x)-m_{2}(x)
\end{align*}
 then a simple calculation shows that 

\begin{align*}
\frac{d}{ds}\int_{R^{n}}(u_{1}(x,s)-u_{2}(x,s))(m_{1}(x,s)-m_{2}(x,s))dx & =-\int_{R^{n}}(F(x,m_{1}(s))-F(x,m_{2}(s)))(m_{1}(x,s)-m_{2}(x,s))dx\\
- & \frac{1}{2\lambda}\int_{R^{n}}(m_{1}(x,s)+m_{2}(x,s))|Du_{1}(x,s)-Du_{2}(x,s)|^{2}dx
\end{align*}
 and the result follows immediately, recalling that $m_{1},m_{2}$
are positive and using the assumption (\ref{eq:7-43}). $\blacksquare$ 
\end{proof}
We now state the 
\begin{prop}
\label{prop6} We make the assumptions of Proposition \ref{prop3}
and (\ref{eq:7-4}), (\ref{eq:7-41}), (\ref{eq:7-42}), (\ref{eq:7-43}).
We then have 

\begin{equation}
U(x,m,t)=\frac{\partial V}{\partial m}(m,t)(x)=u_{mt}(x,t)\label{eq:7-45}
\end{equation}
Moreover, we have the estimate 

\begin{equation}
|U(x,m,t)|\leq C(1+|x|^{2}+\int_{R^{n}}|\xi|^{2}m(\xi)d\xi)\label{eq:7-450}
\end{equation}
\end{prop}
\begin{proof}
We recall the definition of the value function $V(m,t),$ see section
\ref{sub:MEAN-FIELD-TYPE}, and formulas (\ref{eq:5-15}) and (\ref{eq:7-3}).
Let $m_{1}(x)$ be some probability density and the functions $u_{1}(x,s)=u_{m_{1}t}(x,s),$
$m_{1}(x,s)=m_{m_{1}t}(x,s)$ solutions of the system HJB-FP (\ref{eq:5-13}).
The feedback $\hat{v}_{1}(x,s)=-\dfrac{1}{\lambda}Du_{1}(x,s)$ is
optimal for the control problem (\ref{eq:5-9}), (\ref{eq:5-10}),
(\ref{eq:5-11}). The corresponding optimal trajectory, starting from
a deterministic value $x$ is $y_{xm_{1}t}(s).$ The probability density
$m_{\hat{v}_{1}}(x,s)$ corresponding to the feedback $\hat{v}_{1}(x,s)$
is the image of $m_{1}$ by the map $x\rightarrow y_{xm_{1}t}(s)$,
so we can write 

\[
m_{1}(s)=m_{\hat{v}_{1}}(s)=y_{m_{1}t}(s)(.)(m_{1})
\]
 We now consider another initial probability density $m_{2}(x)$ and
the same feedback $\hat{v}_{1}.$ Namely we compute $J_{m_{2}t}(\hat{v}_{1}(.)).$
The probability density at time $s,$ with initial condition at time
$t$ equal to $m_{2}$ and feedback $\hat{v}_{1}$ is $y_{m_{1}t}(s)(.)(m_{2})$
denoted $m_{12}(s)=m_{12}(x,s).$ It is solution of the FP equation 

\begin{align*}
\frac{\partial m_{12}}{\partial s}-\dfrac{1}{\lambda}\text{div}(Du_{1}(x)m_{12}) & =0\\
m_{12}(x,t)=m_{2}(x)
\end{align*}
We can then write 

\begin{align*}
J_{m_{2}t}(\hat{v}_{1}(.)) & =\frac{1}{2\lambda}\int_{t}^{T}\int_{R^{n}}|Du_{1}(x,s)|^{2}m_{12}(x,s)dxds+\int_{t}^{T}\int_{R^{n}}m_{12}(x,s)f(x,m_{12}(s))dxds+\\
+ & \int_{R^{n}}m_{12}(x,T)h(x,m_{12}(T))dx
\end{align*}
 Therefore we have the inequality 

\begin{align}
V(m_{2},t)-V(m_{1},t) & \leq J_{m_{2}t}(\hat{v}_{1}(.))-J_{m_{1}t}(\hat{v}_{1}(.))\label{eq:7-46}\\
= & \frac{1}{2\lambda}\int_{t}^{T}\int_{R^{n}}|Du_{1}(x,s)|^{2}(m_{12}(x,s)-m_{1}(x,s))dxds+\nonumber \\
+ & \int_{t}^{T}\int_{R^{n}}(m_{12}(x,s)f(x,m_{12}(s))-m_{1}(x,s)f(x,m_{1}(s)))dxds+\nonumber \\
+ & \int_{R^{n}}(m_{12}(x,T)h(x,m_{12}(T))-m_{1}(x,T)h(x,m_{1}(T)))dx\nonumber 
\end{align}
 We note that 

\begin{align*}
\frac{\partial(m_{12}-m_{1})}{\partial s}-\dfrac{1}{\lambda}\text{div}(Du_{1}(x)(m_{12}-m_{1})) & =0\\
m_{12}(x,t)-m_{1}(x,t)=m_{2}(x)-m_{1}(x)
\end{align*}
 
\begin{align*}
-\frac{\partial}{\partial s}u_{1}+\frac{1}{2\lambda}|Du_{1}|^{2} & =F(x,m_{1}(s))\\
u_{1}(x,T) & =F_{T}(x,m_{1}(T))
\end{align*}
 hence, as easily seen 

\begin{align*}
\int_{R^{n}}u_{1}(x,t)(m_{2}(x)-m_{1}(x))dx & =\frac{1}{2\lambda}\int_{t}^{T}\int_{R^{n}}|Du_{1}(x,s)|^{2}(m_{12}(x,s)-m_{1}(x,s))dxds+
\end{align*}

\[
+\int_{t}^{T}\int_{R^{n}}F(x,m_{1}(s))(m_{12}(x,s)-m_{1}(x,s))dxds+\int_{R^{n}}F_{T}(x,m_{1}(T))(m_{12}(x,T)-m_{1}(x,T))dx
\]
 Combining with (\ref{eq:7-46}) we can write 

\[
V(m_{2},t)-V(m_{1},t)\leq\int_{R^{n}}u_{1}(x,t)(m_{2}(x)-m_{1}(x))dx+
\]
\[
+\int_{t}^{T}\int_{R^{n}}[m_{12}(x,s)f(x,m_{12}(s))-m_{1}(x,s)f(x,m_{1}(s))-F(x,m_{1}(s))(m_{12}(x,s)-m_{1}(x,s))]dxds+
\]

\[
\int_{R^{n}}[m_{12}(x,T)h(x,m_{12}(T))-m_{1}(x,T)f(x,m_{1}(s))-F_{T}(x,m_{1}(T))(m_{12}(x,T)-m_{1}(x,T))]dx
\]
Recalling that $F(x,m)$ is the functional derivative of $\int_{R^{n}}f(x,m)m(x)dx,$
we can write the above inequality as follows 

\begin{equation}
V(m_{2},t)-V(m_{1},t)\leq\int_{R^{n}}u_{1}(x,t)(m_{2}(x)-m_{1}(x))dx+\label{eq:7-47}
\end{equation}
\[
+\int_{t}^{T}\int_{R^{n}}\int_{R^{n}}\int_{0}^{1}\int_{0}^{1}\theta\frac{\partial F}{\partial m}(x,m_{1}(s)+\theta\mu(m_{12}(s)-m_{1}(s))(\xi)(m_{12}(x,s)-m_{1}(x,s))(m_{12}(\xi,s)-m_{1}(\xi,s))dxd\xi dsd\theta d\mu+
\]
\[
+\int_{R^{n}}\int_{R^{n}}\int_{0}^{1}\int_{0}^{1}\theta\frac{\partial F}{\partial m}(x,m_{1}(T)+\theta\mu(m_{12}(T)-m_{1}(T))(\xi)(m_{12}(x,T)-m_{1}(x,T))(m_{12}(\xi,T)-m_{1}(\xi,T))dxd\xi d\theta d\mu
\]
 Recalling that $m_{12}(s)=y_{m_{1}t}(s)(.)(m_{2})$ and $m_{1}(s)=y_{m_{1}t}(s)(.)(m_{1})$,
we can write for a test function $\varphi(x,\xi)$ 

\[
\int_{R^{n}}\int_{R^{n}}\varphi(x,\xi)(m_{12}(x,s)-m_{1}(x,s))(m_{12}(\xi,s)-m_{1}(\xi,s))dxd\xi=
\]
\[
\int_{R^{n}}\int_{R^{n}}\varphi(y_{xm_{1}t}(s),y_{\xi m_{1}t}(s))(m_{2}(x)-m_{1}(x))(m_{2}(\xi)-m_{1}(\xi))dxd\xi
\]

We introduce a pair of random variables $X_{1},X_{2}$whose marginals
are $m_{1},m_{2}.$We then introduce an independent copy $Y_{1},Y_{2}.$
It is easy to convince oneself that we have the relation 

\[
\int_{R^{n}}\int_{R^{n}}\varphi(y_{xm_{1}t}(s),y_{\xi m_{1}t}(s))(m_{2}(x)-m_{1}(x))(m_{2}(\xi)-m_{1}(\xi))dxd\xi=
\]
\begin{align*}
\int_{0}^{1}\int_{0}^{1}E\,D_{\xi}D_{x}\varphi(y_{X_{1}m_{1}t}(s)+\theta(y_{X_{2}m_{1}t}(s)-y_{X_{1}m_{1}t}(s)),y_{Y_{1}m_{1}t}(s)+\mu(y_{Y_{2}m_{1}t}(s)-y_{Y_{1}m_{1}t}(s)))\\
(y_{X_{2}m_{1}t}(s)-y_{X_{1}m_{1}t}(s))(y_{Y_{2}m_{1}t}(s)-y_{Y_{1}m_{1}t}(s))d\theta d\mu
\end{align*}
 If we have $||D_{\xi}D_{x}\varphi(x,\xi)||\leq C,$then we get 

\[
|\int_{R^{n}}\int_{R^{n}}\varphi(y_{xm_{1}t}(s),y_{\xi m_{1}t}(s))(m_{2}(x)-m_{1}(x))(m_{2}(\xi)-m_{1}(\xi))dxd\xi|\leq|
\]

\[
CE|y_{X_{2}m_{1}t}(s)-y_{X_{1}m_{1}t}(s)||y_{Y_{2}m_{1}t}(s)-y_{Y_{1}m_{1}t}(s)
\]

and from the independence property 

\[
\leq C(E|y_{X_{2}m_{1}t}(s)-y_{X_{1}m_{1}t}(s)|)^{2}\leq CE|y_{X_{2}m_{1}t}(s)-y_{X_{1}m_{1}t}(s)|^{2}
\]
 Using property (\ref{eq:6-14}) we obtain also 

\[
|\int_{R^{n}}\int_{R^{n}}\varphi(y_{xm_{1}t}(s),y_{\xi m_{1}t}(s))(m_{2}(x)-m_{1}(x))(m_{2}(\xi)-m_{1}(\xi))dxd\xi|\leq CE|X_{2}-X_{1}|^{2}
\]
 and since$X_{1},X_{2}$ have an arbitrary correlation, this implies
also 

\[
|\int_{R^{n}}\int_{R^{n}}\varphi(y_{xm_{1}t}(s),y_{\xi m_{1}t}(s))(m_{2}(x)-m_{1}(x))(m_{2}(\xi)-m_{1}(\xi))dxd\xi|\leq CW_{2}^{2}(m_{1},m_{2})
\]
 We may apply this result with $\varphi(x,$$\xi)=\frac{\partial F}{\partial m}(x,m_{1}(s)+\theta\mu(m_{12}(s)-m_{1}(s))(\xi).$
Thanks to assumption (\ref{eq:7-42}) the same result carries over.
Therefore we conclude easily the estimate 

\begin{equation}
V(m_{2},t)-V(m_{1},t)\leq\int_{R^{n}}u_{1}(x,t)(m_{2}(x)-m_{1}(x))dx+CW_{2}^{2}(m_{1},m_{2})\label{eq:7-48}
\end{equation}
Interchanging the role of $m_{1},m_{2}$, we have also 

\begin{align*}
V(m_{1},t)-V(m_{2},t) & \leq\int_{R^{n}}u_{2}(x,t)(m_{1}(x)-m_{2}(x))dx+CW_{2}^{2}(m_{1},m_{2})\\
\leq & \int_{R^{n}}u_{1}(x,t)(m_{1}(x)-m_{2}(x))dx+\int_{R^{n}}(u_{2}(x,t)-u(x_{1},t)(m_{1}(x)-m_{2}(x))dx+CW_{2}^{2}(m_{1},m_{2})
\end{align*}
 and from Proposition \ref{prop5-1} and assumption (\ref{eq:7-43})
the 2nd integral is negative, which implies 

\[
V(m_{1},t)-V(m_{2},t)\leq\int_{R^{n}}u_{1}(x,t)(m_{1}(x)-m_{2}(x))dx+CW_{2}^{2}(m_{1},m_{2})
\]
or

\[
V(m_{2},t)-V(m_{1},t)\geq\int_{R^{n}}u_{1}(x,t)(m_{2}(x)-m_{1}(x))dx-CW_{2}^{2}(m_{1},m_{2})
\]
 and comparing with (\ref{eq:7-48}) we can assert 

\begin{equation}
|V(m_{2},t)-V(m_{1},t)-\int_{R^{n}}u_{1}(x,t)(m_{2}(x)-m_{1}(x))dx|\leq CW_{2}^{2}(m_{1},m_{2})\label{eq:7-49}
\end{equation}
Now we have 

\begin{equation}
|u_{mt}(x,t)|\leq C(1+|x|^{2}+\int_{R^{n}}|\xi|^{2}m(\xi)d\xi\label{eq:7-50}
\end{equation}
So for any curve $m_{\epsilon}\in\mathcal{P}_{2}$, $u_{mt}(x,t)\in L_{m_{\epsilon}}^{1}$.
From the estimate (\ref{eq:7-49}) we get immediately the result (\ref{eq:7-45}).
The proof has been completed. $\blacksquare$
\end{proof}

\subsection{OBTAINING BELLMAN EQUATION}

We have seen in section \ref{sub:INTERPRETATION-OF-THE}that 

\[
z_{xmt}(s)=Du_{mt}(y_{xmt}(s),s)
\]
 and thus 

\begin{equation}
D_{x}U(x,m,t)=z_{xmt}(t)\label{eq:7-8}
\end{equation}
Therefore from the estimate (\ref{eq:6-15}) we can assert that 

\begin{equation}
|D_{x}U(x,m,t)|\leq C[1+|x|+\sqrt{\int|\xi|^{2}m(\xi)d\xi}]\label{eq:7-9}
\end{equation}
In particular, we can see that $D_{x}U(x,m,t)$ belongs to $L_{m}^{2}(R^{n};R^{n}).$
But then, recalling the correspondance $V(X,t)=V(m,t)|_{m=\mathcal{L}_{X}},$we
can write 

\begin{equation}
D_{X}V(X,t)=D_{x}U(X,\mathcal{L}_{X},t)\label{eq:7-10}
\end{equation}
and 

\[
||D_{X}V(X,t)||^{2}=\int_{R^{n}}|D_{x}\frac{\partial V(m,t)}{\partial m}(x)|^{2}m(x)dx
\]
If we look at the Bellman equation in the Hilbert space $\mathcal{H}$,
see (\ref{eq:2-8}) we obtain exactly (\ref{eq:5-16}). So we can
state 
\begin{prop}
\label{prop5} We make the assumptions of Proposition \ref{prop6}.
The value function $V(m,t)$ of the problem (\ref{eq:5-9}),(\ref{eq:5-12})
or equivalently (\ref{eq:5-10}),(\ref{eq:5-11}) satisfies the estimates
(\ref{eq:7-2}), (\ref{eq:7-450}) (with $U(x,m,t)=\dfrac{\partial V(m,t)}{\partial m}(x)$
) and (\ref{eq:7-9}). It is the unique solution, satisfying these
estimates, of the Bellman equation (\ref{eq:5-16}). Moreover, we
have the explicit formula (\ref{eq:7-1}) with $y_{xmt}(s)$ being
the unique solution of (\ref{eq:6-9}) and $z_{xmt}(s)=-\lambda\dfrac{dy_{xmt}(s)}{ds}$ 
\end{prop}

\subsection{OBTAINING THE SCALAR MASTER EQUATION}

We can derive the scalar master equation from the probabilistic master
equation (\ref{eq:3-9}), which we write as follows 

\begin{align}
\frac{\partial\mathcal{U}}{\partial t}-\frac{1}{2\lambda}D_{X}||\mathcal{U}(X,t)||^{2}+D_{X}\mathcal{F}(X) & =0\label{eq:3-9-1}\\
\mathcal{U}(X,T)=D_{X}\mathcal{F}_{T}(X)\nonumber 
\end{align}
 We know that $\mathcal{U}(X,t)=D_{x}U(X,\mathcal{L}_{X},t)$ and
$U(x,m,t)=u_{mt}(x,t)$ and $D_{x}u_{mt}(x,t)=z_{xmt}(t).$ Therefore
$||\mathcal{U}(X,t)||^{2}=\int_{R^{n}}|D_{\xi}U(\xi,m,t)|^{2}m(\xi)d\xi.$
Since this functional of $m$ only has a derivative in the Hilbert
space $\mathcal{H}$ it can be written as follows 

\[
D_{X}||\mathcal{U}(X,t)||^{2}=D_{x}\frac{\partial}{\partial m}\,(\int_{R^{n}}|D_{\xi}U(\xi,m,t)|^{2}m(\xi)d\xi)(X)
\]
 Recalling that $D_{X}\mathcal{F}(X)=D_{x}F(X,m),$ $D_{X}\mathcal{F}_{T}(X)=D_{x}F_{T}(X,m),$
we see that (\ref{eq:3-9-1}) can be wriiten as follows

\begin{align*}
D_{x}[\frac{\partial U(X,m,t)}{\partial t}-\frac{1}{2\lambda}\frac{\partial}{\partial m}\,(\int_{R^{n}}|D_{\xi}U(\xi,m,t)|^{2}m(\xi)d\xi)(X)+F(X,m)] & =0\\
D_{x}U(X,m,T)=D_{x}F_{T}(X,m)
\end{align*}
This leads to 

\begin{align*}
\frac{\partial U(x,m,t)}{\partial t}-\frac{1}{2\lambda}\frac{\partial}{\partial m}\,(\int_{R^{n}}|D_{\xi}U(\xi,m,t)|^{2}m(\xi)d\xi)(x)+F(x,m) & =0\\
U(x,m,T)=F_{T}(x,m)
\end{align*}
 which we can write as (\ref{eq:5-17}), taking account of the symmetry
property $\dfrac{\partial}{\partial m}U(x,m,t)(\xi)=\dfrac{\partial}{\partial m}U(\xi,m,t)(x)$.
$\blacksquare$ 

This proof is not fully rigorous. It assumes implicitly the existence
of $\dfrac{\partial}{\partial m}U(x,m,t)(\xi),$ which is the 2nd
derivative of the function $V(m,t).$ To study it rigorously and give
an implicit formula for $\dfrac{\partial}{\partial m}U(x,m,t)(\xi)$
, one can use the system of HJB-FP equations (\ref{eq:5-13}) and
write the solution as $u_{mt}(x,s),m_{mt}(x,s)$ to emphasize the
initial conditions $m,t.$ We then consider the functional derivatives
$\dfrac{\partial u_{mt}(x,s)}{\partial m}(\xi)$, $\dfrac{\partial m_{mt}(x,s)}{\partial m}(\xi)$
and differentiate formally the system of HJB-FP equations. To simplify
notation, we take a test function $\tilde{m}(\xi)$ and consider 

\[
\int_{R^{n}}\dfrac{\partial u_{mt}(x,s)}{\partial m}(\xi)\tilde{m}(\xi)\,d\xi,\:\int_{R^{n}}\dfrac{\partial m_{mt}(x,s)}{\partial m}(\xi)\tilde{m}(\xi)\,d\xi
\]
 which we note $\tilde{u}_{mt;\tilde{m}}(x,s),\:$ $\tilde{m}_{mt;\tilde{m}}(x,s)$
and to simplify further $\tilde{u}(x,s),$ $\tilde{m}(x,s).$ In particular
, $\tilde{u}(x,t)=\int\dfrac{\partial}{\partial m}U(x,m,t)(\xi)\tilde{m}(\xi)\,d\xi$.
The pair $\tilde{u}(x,s),\tilde{m}(x,s)$ is solution of a system
of linear P.D.E. as follows 

\begin{align}
-\dfrac{\partial\tilde{u}}{\partial s}+\frac{1}{\lambda}D\tilde{u}.Du & =\int\frac{\partial F}{\partial m}(x,m(s))(\xi)\tilde{m}(\xi,s)d\xi\label{eq:7-11}\\
\tilde{u}(x,T) & =\int\frac{\partial F_{T}}{\partial m}(x,m(T))(\xi)\tilde{m}(\xi,T)d\xi\nonumber 
\end{align}
\begin{align*}
\dfrac{\partial\tilde{m}}{\partial s}-\frac{1}{\lambda}\text{div}(D\tilde{u}\,m+Du\,\tilde{m}) & =0\\
\tilde{m}(x,t)=\tilde{m}(x)
\end{align*}
 This system is obtained by linearization of the system (\ref{eq:5-13}).
The functions $u(x,s),$$m(x,s)$ are solutions of the system (\ref{eq:5-13}).
We can write also

\begin{equation}
\tilde{u}(x,s)=\int_{s}^{T}\int_{R^{n}}\frac{\partial F}{\partial m}(y_{xmt}(\sigma),y_{mt}(\sigma)(.)(m))(\xi)\tilde{m}(\xi,\sigma)d\xi d\sigma+\int_{R^{n}}\frac{\partial F_{T}}{\partial m}(y_{xmt}(T),y_{mt}(T)(.)(m))(\xi)\tilde{m}(\xi,T)d\xi\label{eq:7-12}
\end{equation}

We can then study (\ref{eq:7-12}) as a fixed point equation in the
function $\tilde{u}(x,s).$

\section{\label{sec:QUADRATIC-CASE}QUADRATIC CASE}

\subsection{ASSUMPTIONS }

We consider the quadratic case, (\ref{eq:5-53}). We also take 

\begin{equation}
h(x,m)=\frac{1}{2}(x-S_{T}\bar{x})^{*}\bar{Q_{T}}(x-S_{T}\bar{x})+\frac{1}{2}x^{*}Q_{T}x\label{eq:5-13-1}
\end{equation}
In the space $\mathcal{H}$ we have 

\begin{align}
\mathcal{F}(X) & =\frac{1}{2}EX^{*}(Q+\bar{Q})X+\frac{1}{2}EX^{*}(S^{*}\bar{Q}S-\bar{Q}S-S^{*}\bar{Q})EX\nonumber \\
\mathcal{F}_{T}(X) & =\frac{1}{2}EX^{*}(Q_{T}+\bar{Q})X+\frac{1}{2}EX^{*}(S^{*}\bar{Q}S-\bar{Q}S-S^{*}\bar{Q})EX\label{eq:8-1}
\end{align}
 We can write $\mathcal{F}(X)=Ef(X,\mathcal{L}_{X})$ with 

\begin{equation}
f(x,m)=\frac{1}{2}(x-S\bar{x})^{*}\bar{Q}(x-S\bar{x})+\frac{1}{2}x^{*}Qx\label{eq:8-2}
\end{equation}
 in which $\bar{x}=\int xm(x)dx$, assuming the probability law of
$X$ has a density, $m.$ So we can also write 

\begin{align*}
\mathcal{F}(X) & =\Phi(m)=\dfrac{1}{2}\int_{R^{n}}x*(Q+\bar{Q)}xm(x)dx+\dfrac{1}{2}\int_{R^{n}}x^{*}m(x)dx\,(S^{*}\bar{Q}S-\bar{Q}S-S^{*}\bar{Q})\,\int_{R^{n}}xm(x)dx\\
 & =\int_{R^{n}}f(x,m)m(x)dx
\end{align*}
 If we take $m\in L^{2}(R^{n})\cap L^{1}(R^{n})$ not necessarily
a probability density, then we have to introduce $m_{1}=\int_{R^{n}}m(x)dx$
and write 

\begin{equation}
\Phi(m)=\int_{R^{n}}f(x,m)m(x)dx=\label{eq:8-3}
\end{equation}

\[
=\dfrac{1}{2}\int_{R^{n}}x*(Q+\bar{Q)}xm(x)dx+\dfrac{1}{2}\bar{x}^{*}S^{*}\bar{Q}S\bar{x}\,m_{1}-\dfrac{1}{2}\bar{x}^{*}(\bar{Q}S+S^{*}\bar{Q})\bar{x}
\]
 We have noted $F(x,m)=\dfrac{\partial\Phi(m)}{\partial m}(x).$ Then
as a Gâteaux differential we have 

\begin{equation}
F(x,m)=\dfrac{1}{2}x*(Q+\bar{Q)}x+\bar{x}^{*}(S^{*}\bar{Q}S\,m_{1}-\bar{Q}S-S^{*}\bar{Q})x+\dfrac{1}{2}\bar{x}^{*}S^{*}\bar{Q}S\bar{x}\label{eq:8-4}
\end{equation}

We note that 

\begin{equation}
D_{X}\mathcal{F}(X)=(Q+\bar{Q})X+(S^{*}\bar{Q}S-\bar{Q}S-S^{*}\bar{Q})EX\label{eq:8-5}
\end{equation}

So the equality $D_{X}\mathcal{F}(X)=D_{x}F(X,m)$ is true only when
$m_{1}=1.$ It is important to keep in mind that when we work with
Gâteaux differentials, we have to make calculations with the term
$m_{1},$even though that eventually, when applied to $m=$ probability
density, we shall have $m_{1}=1.$ To understand further this point,
let us compute the 2nd derivative. We have 

\begin{equation}
\frac{\partial F}{\partial m}(x,m)(\xi)=\bar{x}^{*}S^{*}\bar{Q}S(x+\xi)+\xi^{*}(S^{*}\bar{Q}S\,m_{1}-\bar{Q}S-S^{*}\bar{Q})x\label{eq:8-6}
\end{equation}
We see that this formula is symmetric in $x,\xi$ as needed. Without
the term $m_{1}$ in (\ref{eq:8-4}) this will not be true. We have 

\[
D_{x}^{2}F(x,m)=Q+\bar{Q},\;D_{x}D_{\xi}\frac{\partial F}{\partial m}(x,m)(\xi)=(S^{*}\bar{Q}S\,m_{1}-\bar{Q}S-S^{*}\bar{Q})
\]
 and , see \cite{BFY2} 

\begin{align*}
D_{X}^{2}\mathcal{F}(X)Z & =D_{x}^{2}F(X,\mathcal{L}_{X})Z+E_{Y\tilde{Z}}D_{x}D_{y}\frac{\partial F}{\partial m}(X,\mathcal{L}_{X})(Y)\tilde{Z}=\\
 & =(Q+\bar{Q})Z+(S^{*}\bar{Q}S-\bar{Q}S-S^{*}\bar{Q})EZ
\end{align*}
 which is exactly what we obtain by differentiating (\ref{eq:8-5})
in the Hilbert space.

\subsection{BELLMAN EQUATION}

Bellman equation ( \ref{eq:5-16}) writes 

\begin{equation}
\frac{\partial V}{\partial t}-\frac{1}{2\lambda}\int_{R^{n}}|D_{\xi}\frac{\partial V(m,t)}{\partial m}(\xi)|^{2}m(\xi)d\xi+\label{eq:8-7}
\end{equation}
\[
+\dfrac{1}{2}\int_{R^{n}}\xi*(Q+\bar{Q)}\xi m(\xi)d\xi+\dfrac{1}{2}\bar{x}^{*}S^{*}\bar{Q}S\bar{x}\,m_{1}-\dfrac{1}{2}\bar{x}^{*}(\bar{Q}S+S^{*}\bar{Q})\bar{x}=0
\]

\[
V(m,T)=\dfrac{1}{2}\int_{R^{n}}\xi*(Q_{T}+\bar{Q_{T}})\xi m(\xi)d\xi+\dfrac{1}{2}\bar{x}^{*}S_{T}^{*}\bar{Q}_{T}S_{T}\bar{x}\,m_{1}-\dfrac{1}{2}\bar{x}^{*}(\bar{Q}_{T}S_{T}+S_{T}^{*}\bar{Q}_{T})\bar{x}
\]
The solution is 

\begin{equation}
V(m,t)=\dfrac{1}{2}\int_{R^{n}}\xi*P(t)\xi m(\xi)d\xi+\dfrac{1}{2}\bar{x}^{*}\Sigma(t;m_{1})\bar{x}\label{eq:8-8}
\end{equation}
 with 

\begin{align}
\frac{dP}{dt}-\frac{P^{2}}{\lambda}+Q+\bar{Q} & =0\label{eq:8-9}\\
P(T)=Q_{T}+\bar{Q_{T}}\nonumber 
\end{align}
\begin{align}
\frac{d\Sigma}{dt}-\frac{1}{\lambda}(\Sigma P+P\Sigma)-\frac{1}{\lambda}\Sigma^{2}m_{1}+S^{*}\bar{Q}S\,m_{1}-(\bar{Q}S+S^{*}\bar{Q}) & =0\label{eq:8-10}\\
\Sigma(T;m_{1})=S_{T}^{*}\bar{Q}_{T}S_{T}\,m_{1}-(\bar{Q}_{T}S_{T}+S_{T}^{*}\bar{Q}_{T})\nonumber 
\end{align}

\subsection{MASTER EQUATION}

The scalar Master equation (\ref{eq:5-17}) reads 

\begin{align}
\frac{\partial U}{\partial t}-\frac{1}{\lambda}\int_{R^{n}}D_{\xi}\frac{\partial U}{\partial m}(x,m,t)(\xi).D_{\xi}U(\xi,m,t)m(\xi)d\xi-\frac{1}{2\lambda}|D_{x}U(x,m,t)|^{2}\label{eq:5-17-1}\\
+\dfrac{1}{2}x*(Q+\bar{Q)}x+\bar{x}^{*}(S^{*}\bar{Q}S\,m_{1}-\bar{Q}S-S^{*}\bar{Q})x+\dfrac{1}{2}\bar{x}^{*}S^{*}\bar{Q}S\bar{x}=0\nonumber \\
U(x,m,T)=\dfrac{1}{2}x*(Q_{T}+\bar{Q}_{T})x+\bar{x}^{*}(S_{T}^{*}\bar{Q_{T}}S_{T}\,m_{1}-\bar{Q_{T}}S_{T}-S_{T}^{*}\bar{Q}_{T})x+\dfrac{1}{2}\bar{x}^{*}S_{T}^{*}\bar{Q_{T}}S_{T}\bar{x}\nonumber 
\end{align}
Its solution is

\begin{equation}
U(x,m,t)=\frac{\partial V(m,t)}{\partial m}(x)=\frac{1}{2}x^{*}P(t)x+\bar{x}^{*}\Sigma(t;m_{1})x+\frac{1}{2}\bar{x}^{*}\frac{\partial\Sigma(t;m_{1})}{\partial m_{1}}\bar{x}\label{eq:5-18-1}
\end{equation}
We have 

\begin{equation}
D_{x}U(x,m,t)=P(t)x+\Sigma(t;m_{1})\bar{x}\label{eq:5-19-1}
\end{equation}

\[
\frac{\partial U}{\partial m}(x,m,t)(\xi)=\bar{x}^{*}\frac{\partial\Sigma(t;m_{1})}{\partial m_{1}}(x+\xi)+\xi^{*}\Sigma(t;m_{1})x+\frac{1}{2}\bar{x}^{*}\frac{\partial^{2}\Sigma(t;m_{1})}{\partial m_{1}^{2}}\bar{x}
\]
\[
D_{\xi}\frac{\partial U}{\partial m}(x,m,t)(\xi)=\frac{\partial\Sigma(t;m_{1})}{\partial m_{1}}\bar{x}+\Sigma(t;m_{1})x
\]
We note that $\Gamma(t;m_{1})=\dfrac{\partial\Sigma(t;m_{1})}{\partial m_{1}}$
satisfies the equation 

\begin{align}
\frac{d\Gamma}{dt}-\frac{1}{\lambda}(\Gamma(P+\Sigma m_{1})+(P+\Sigma m_{1})\Gamma)-\frac{1}{\lambda^{2}}\Sigma^{2}+S^{*}\bar{Q}S & =0\label{eq:5-19-2}\\
\Gamma(T;m_{1})=S_{T}^{*}\bar{Q}_{T}S_{T}\nonumber 
\end{align}
 and we check easily that the function $U(x,m,t)$ defined by (\ref{eq:5-18-1})
is solution of the scalar master equation (\ref{eq:5-17-1}). 

We turn to the vector master equation (\ref{eq:5-20}) which reads 

\begin{align}
\frac{\partial\mathcal{U}}{\partial t}-\frac{1}{\lambda}\int_{R^{n}}D_{\xi}\frac{\partial\mathcal{U}}{\partial m}(x,m,t)(\xi)\,\mathcal{U}(\xi,m,t)m(\xi)d\xi-\frac{1}{\lambda}D_{x}\mathcal{U}(x,m,t)\,\mathcal{U}(x,m,t)+\label{eq:5-20-1}\\
+(Q+\bar{Q})x+(S^{*}\bar{Q}Sm_{1}-\bar{Q}S-S^{*}\bar{Q})\bar{x}=0\nonumber 
\end{align}
\[
\mathcal{U}(x,m,T)=(Q_{T}+\bar{Q_{T}})x+(S_{T}^{*}\bar{Q_{T}}S_{T}m_{1}-\bar{Q}_{T}S_{T}-S_{T}^{*}\bar{Q_{T}})\bar{x}
\]
 whose solution is 

\begin{equation}
\mathcal{U}(x,m,t)=D_{x}U(x,m,t)=P(t)x+\Sigma(t;m_{1})\bar{x}\label{eq:5-21}
\end{equation}
This statement is easily verified.

\subsection{SYSTEM OF HJB-FP EQUATIONS}

We now look at the system (\ref{eq:5-13}) which reads 

\begin{align}
-\frac{\partial u}{\partial s}+\frac{1}{2\lambda}|Du|^{2} & =\dfrac{1}{2}x*(Q+\bar{Q)}x+\bar{x}^{*}(s)(S^{*}\bar{Q}S\,m_{1}(s)-\bar{Q}S-S^{*}\bar{Q})x+\dfrac{1}{2}\bar{x}^{*}(s)S^{*}\bar{Q}S\bar{x}(s)\label{eq:5-13-2}\\
u(x,T) & =\dfrac{1}{2}x*(Q_{T}+\bar{Q}_{T})x+\bar{x}^{*}(T)(S_{T}^{*}\bar{Q}_{T}S_{T}\,m_{1}(T)-\bar{Q}_{T}S_{T}-S_{T}^{*}\bar{Q}_{T})x+\dfrac{1}{2}\bar{x}^{*}(T)S_{T}^{*}\bar{Q}_{T}S_{T}\bar{x}(T)\nonumber \\
\frac{\partial m}{\partial s}-\frac{1}{\lambda}\text{div} & (Du\,m)=0\nonumber \\
m(x,t) & =m(x)\nonumber 
\end{align}
 It is immediate to see that $m_{1}(s)=m_{1}.$The function $\bar{x}(s)$
represents the mean $\int_{R^{n}}\xi m(\xi,s)d\xi.$We do not need
to obtain the full probability $m(x,s).$The mean is sufficient. One
can check the formula

\begin{equation}
u(x,s)=\frac{1}{2}x^{*}P(s)x+x^{*}\Sigma(s;m_{1})\bar{x}(s)+\frac{1}{2}\bar{x}(s)^{*}\frac{\partial\Sigma(s;m_{1})}{\partial m_{1}}\bar{x}(s)\label{eq:5-14-1}
\end{equation}

In particular $u(x,t)=U(x,m,t)$ given by (\ref{eq:5-18-1}). Also
$u(x,s)=U(x,m(s),s).$ Note that $\bar{x}(s)$ evolves as follows 

\begin{align}
\frac{d\bar{x}}{ds}+\frac{1}{\lambda}(P(s)+m_{1}\Sigma(s;m_{1}))\bar{x}(s) & =0\label{eq:5-15-1}\\
\bar{x}(t)=\bar{x}\nonumber 
\end{align}
 We have seen that $U(x,m,t)$ is differentiable in $m$ with 

\[
\frac{\partial U}{\partial m}(x,m,t)(\xi)=\bar{x}^{*}\frac{\partial\Sigma(t;m_{1})}{\partial m_{1}}(x+\xi)+\xi^{*}\Sigma(t;m_{1})x+\frac{1}{2}\bar{x}^{*}\frac{\partial^{2}\Sigma(t;m_{1})}{\partial m_{1}^{2}}\bar{x}
\]
 If we consider a test function $\tilde{m}(\xi)$ and define 

\[
\tilde{u}(x,t)=\tilde{u}_{mt;\tilde{m}}(x,t)=\lim_{\theta\rightarrow0}\frac{u_{m+\theta\tilde{m},t}(x,t)-u_{mt}(x,t)}{\theta}
\]
 we get 

\begin{equation}
\tilde{u}(x,t)=x^{*}\Sigma(t;m_{1})\tilde{\bar{x}}+\bar{x}^{*}\frac{\partial\Sigma(t;m_{1})}{\partial m_{1}}\tilde{\bar{x}}+\tilde{m}_{1}(\bar{x}^{*}\frac{\partial\Sigma(t;m_{1})}{\partial m_{1}}x+\frac{1}{2}\bar{x}^{*}\frac{\partial^{2}\Sigma(t;m_{1})}{\partial m_{1}^{2}}\bar{x})\label{eq:5-16-1}
\end{equation}
We can also compute 

\[
\tilde{u}(x,s)=\tilde{u}_{mt;\tilde{m}}(x,s)=\lim_{\theta\rightarrow0}\frac{u_{m+\theta\tilde{m},t}(x,s)-u_{mt}(x,s)}{\theta}
\]
 We have 

\begin{equation}
\tilde{u}(x,s)=x^{*}(\Sigma(s;m_{1})\tilde{\bar{x}}(s)+\frac{\partial\Sigma(s;m_{1})}{\partial m_{1}}\bar{x}(s)\tilde{m}_{1})+\bar{x}^{*}(s)\frac{\partial\Sigma(s;m_{1})}{\partial m_{1}}\tilde{\bar{x}}(s)+\tilde{m}_{1}\frac{1}{2}\bar{x}^{*}(s)\frac{\partial^{2}\Sigma(s;m_{1})}{\partial m_{1}^{2}}\bar{x}(s)\label{eq:8-1-1}
\end{equation}
in which 

\begin{equation}
\frac{d\tilde{\bar{x}}(s)}{ds}+\frac{1}{\lambda}(P(s)+m_{1}\Sigma(s;m_{1}))\tilde{\bar{x}}(s)+\frac{1}{\lambda}\tilde{m}_{1}(\Sigma(s;m_{1})+m_{1}\frac{\partial\Sigma(s;m_{1})}{\partial m_{1}})\bar{x}(s)=0\label{eq:8-2-1}
\end{equation}
and where $\tilde{m}_{1}=\int_{R^{n}}\tilde{m}(\xi)d\xi.$ The function
$\tilde{u}(x,s)$ is the solution of the linearized equation (\ref{eq:7-11}),
namely 

\begin{align}
-\dfrac{\partial\tilde{u}}{\partial s}+\frac{1}{\lambda}D\tilde{u}.(P(s)x+\Sigma(s;m_{1})\bar{x}(s)) & =x^{*}\left(S^{*}\bar{Q}S\,\bar{x}(s)\tilde{m}_{1}+(S^{*}\bar{Q}S\,m_{1}-\bar{Q}S-S^{*}\bar{Q})\tilde{\bar{x}}(s)\right)+\label{eq:8-3-1}\\
+ & \bar{x}(s)^{*}S^{*}\bar{Q}S\tilde{\bar{x}}(s)\nonumber 
\end{align}
\begin{align*}
\tilde{u}(x,T) & =x^{*}\left(S^{*}\bar{Q}S\,\bar{x}(T)\tilde{m}_{1}+(S^{*}\bar{Q}S\,m_{1}-\bar{Q}S-S^{*}\bar{Q})\tilde{\bar{x}}(T)\right)+\\
+ & \bar{x}(T)^{*}S^{*}\bar{Q}S\tilde{\bar{x}}(T)
\end{align*}
 which can be checked by direct calculation.

\subsection{STATE EQUATION }

We consider equation (\ref{eq:6-9}) in the quadratic case. Since
we know the function $u_{mt}(x,s)$ see (\ref{eq:5-14-1}) the best
is to use the fact that $y_{xmt}(s)$ is the solution of 

\begin{align*}
\frac{dy}{ds} & =-\frac{1}{\lambda}Du(y(s),s)\\
y(t) & =x
\end{align*}
 We get the explicit solution 

\begin{equation}
y_{xmt}(s)=\exp-\frac{1}{\lambda}\int_{t}^{s}P(\sigma)d\sigma\,x-\frac{1}{\lambda}\int_{t}^{s}(\exp-\frac{1}{\lambda}\int_{\sigma}^{s}P(\tau)d\tau\,\Sigma(\sigma)\exp-\frac{1}{\lambda}\int_{t}^{\sigma}(\Sigma(\tau)+P(\tau))d\tau)d\sigma\,\bar{x}]\label{eq:8-4-1}
\end{equation}

\subsection{FORMULATION IN THE HILBERT SPACE }

We can formulate Bellman equation and the Master equation in the Hilbert
space $\mathcal{H}$ . We have first Bellman equation 

\begin{equation}
\frac{\partial V}{\partial t}-\frac{1}{2\lambda}||D_{X}V||^{2}+\frac{1}{2}EX^{*}(Q+\bar{Q})X+\frac{1}{2}EX^{*}(S^{*}\bar{Q}S-\bar{Q}S-S^{*}\bar{Q})EX=0\label{eq:8-5-1}
\end{equation}
\[
V(X,T)=\frac{1}{2}EX^{*}(Q_{T}+\bar{Q}_{T})X+\frac{1}{2}EX^{*}(S_{T}^{*}\bar{Q}_{T}S_{T}-\bar{Q}_{T}S_{T}-S_{T}^{*}\bar{Q}_{T})EX
\]
 whose solution is 

\begin{equation}
V(X,t)=\frac{1}{2}EX^{*}P(t)X+\frac{1}{2}EX^{*}\Sigma(t)EX\label{eq:8-6-1}
\end{equation}
with $\Sigma(t)=\Sigma(t;1).$ The Master equation reads

\begin{equation}
\frac{\partial\mathcal{U}}{\partial t}-\frac{1}{\lambda}D_{X}\mathcal{U}(X,t)\,\mathcal{U}(X,t)+(Q+\bar{Q})X+(S^{*}\bar{Q}S-\bar{Q}S-S^{*}\bar{Q})EX=0\label{eq:8-7-1}
\end{equation}
\[
\mathcal{U}(X,T)=(Q_{T}+\bar{Q}_{T})X+(S_{T}^{*}\bar{Q}_{T}S_{T}-\bar{Q}_{T}S_{T}-S_{T}^{*}\bar{Q}_{T})EX
\]
 whose solution is $\mathcal{U}(X,t)=P(t)X+\Sigma(t)EX.$ Note that
$D_{X}\mathcal{U}(X,t)Z=P(t)Z+\Sigma(t)EZ$ . The state equation is
the solution of 

\begin{equation}
\frac{dY}{ds}=-\frac{1}{\lambda}(P(s)Y(s)+\Sigma(s)EY(s))\label{eq:8-8-1}
\end{equation}
\[
Y(t)=X
\]
 hence the formula 

\begin{align}
Y(s) & =\exp-\frac{1}{\lambda}\int_{t}^{s}P(\sigma)d\sigma X-\label{eq:8-9-1}\\
- & \int_{t}^{s}(\exp-\frac{1}{\lambda}\int_{\sigma}^{s}P(\tau)d\tau\,\Sigma(\sigma)\,\exp-\frac{1}{\lambda}\int_{t}^{\sigma}(P(\tau)+\Sigma(\tau))d\tau)d\sigma)EX\nonumber 
\end{align}

\end{document}